\documentclass{siamart}

\usepackage{lipsum}
\usepackage{amsfonts}
\usepackage{graphicx,epstopdf}
\usepackage[caption=false]{subfig}
\usepackage{epstopdf}
\usepackage{algorithmic}
\ifpdf
  \DeclareGraphicsExtensions{.eps,.pdf,.png,.jpg}
\else
  \DeclareGraphicsExtensions{.eps}
\fi

\newcommand{\TheTitle}{A posteriori error estimation for the $p$-curl problem} 
\newcommand{\TheAuthors}{Andy T. S. Wan, and Marc Laforest}

\headers{\TheTitle}{\TheAuthors}

\title{{\TheTitle}\thanks{This work was supported by FQNRT, NSERC and MITACS.}}

\author{
  Andy T. S. Wan\thanks{Department of Mathematics and Statistics, University of Northern British Columbia, Prince George, British Columbia, Canada
    (\email{andy.wan@unbc.ca}).}
  \and
  Marc Laforest\thanks{d\'epartement de math\'ematiques et g\'enie industriel, \'Ecole Polytechnique de Montr\'eal, Montr\'eal, Qu\'ebec, Canada (\email{marc.laforest@polymtl.ca}).}
}

\usepackage{amsopn}

\usepackage{amstext,amsfonts,amsmath,stmaryrd,amssymb}

\newsiamremark{hypothesis}{Hypothesis}

\providecommand{\norm}[1]{\left\lVert#1\right\rVert}
\providecommand{\pair}[1]{\langle#1\rangle_{\Omega}}
\providecommand{\pairh}[1]{\langle#1\rangle_{\Omega_h}}

\providecommand{\rpair}[1]{\left(#1\right)_{\Omega}}
\providecommand{\rpairh}[1]{\left(#1\right)_{\Omega_h}}
\providecommand{\rpairk}[1]{\left(#1\right)_{K}}
\providecommand{\rpaire}[1]{\left(#1\right)_{F}}
\providecommand{\jump}[1]{\llbracket#1\rrbracket}
\providecommand{\Tau}{\mathcal{T}}

\newcommand{\bb}{\boldsymbol}
\renewcommand{\L}[1]{L^{#1}}
\newcommand{\Ldom}[1]{L^{#1}(\Omega)}

\newcommand{\W}[1]{W^{#1,p}(\Omega)}
\newcommand{\Wz}{W_0^{1,p}(\Omega)}
\newcommand{\Wcurl}{W^p(\text{curl};\Omega)}
\newcommand{\Wcurlh}{W^p(\text{curl};\Omega_h)}
\newcommand{\Wzcurl}{W^p_0(\text{curl};\Omega)}
\newcommand{\Wzcurlh}{W^p_0(\text{curl};\Omega_h)}
\newcommand{\Wcurlz}{W^p(\text{curl}^0;\Omega)}
\newcommand{\Wdiv}{W^p(\text{div};\Omega)}

\newcommand{\Wdivz}{W^p(\text{div}^0;\Omega)}
\newcommand{\X}{V^p(\Omega)}
\newcommand{\Xq}{V^q(\Omega)}
\newcommand{\Xh}{V_{h}}
\newcommand{\Xhz}{V_{h,0}}
\newcommand{\curl}{\nabla\times}
\newcommand{\grad}{\nabla}
\renewcommand{\div}{\nabla\cdot}
\newcommand{\pdt}{\partial_t}
\newcommand{\param}{\alpha}

\newcommand{\Res}{\operatorname{Res}}

\newcommand{\NC}{\operatorname{NC}}

\ifpdf
\hypersetup{
  pdftitle={\TheTitle},
  pdfauthor={\TheAuthors}
}
\fi

\begin{document}

\maketitle

\begin{abstract}
We derive a posteriori error estimates for a semi-discrete finite element approximation of a nonlinear eddy current problem arising from applied superconductivity, known as the $p$-curl problem. In particular, we show the reliability for non-conforming N\'{e}d\'{e}lec elements based on a residual type argument and a Helmholtz-Weyl decomposition of $\Wzcurl$. As a consequence, we are also able to derive an a posteriori error estimate for a quantity of interest called the AC loss. The nonlinearity for this form of Maxwell's equation is an analogue of the one found in the $p$-Laplacian. It is handled without linearizing around the approximate solution. The non-conformity is dealt by adapting error decomposition techniques of Carstensen, Hu and Orlando. Geometric non-conformities also appear because the continuous problem is defined over a bounded $C^{1,1}$ domain while the discrete problem is formulated over a weaker polyhedral domain. The semi-discrete formulation studied in this paper is often encountered in commercial codes and is shown to be well-posed. The paper concludes with numerical results confirming the reliability of the a posteriori error estimate.
\end{abstract}



\section{Introduction}
The optimal design of the next generation of high-temperat-\\ure superconductor (HTS) devices will require fast and accurate approximations of the time-dependent magnetic field 
inside complex domains \cite{GriAE14}. Potential devices include, among others, passive current-fault limiters, MagLev trains and power links in the CERN accelerator.
In a superconductor, any reversal of variation rate in the magnetic field generates a strong front in the current density profile, as well as a discontinuity
in the magnetic field profile, which is not traditionally encountered in computational electromagnetism. It is therefore clear that a posteriori error estimators 
can play an important role in the simulation of such devices; first to achieve design tolerances and secondly to implement adaptive mesh refinement.

At power frequencies of the applications concerned, and when the operating conditions are such that we do not exceed significantly
the critical current of superconducting wires, the eddy current problem with the so-called power-law model for the resistivity
adequately describes the evolution of the magnetic 
field $\bb{u} = \bb{u} (t, \bb{x} )$ for $(t, \bb{x} ) \in I \times \Omega \subset \mathbb{R}^+ \times \mathbb{R}^3$ by 
\begin{align}
   \pdt\bb{u}+\curl \left[ \rho(\curl \bb{u})\curl \bb{u}\right]&=\bb{f},  \hspace{4mm} \mbox{ in } I\times\Omega,   \label{basic_model}  \\
     \nabla \cdot \bb{u} &= 0,  \hspace{4mm}\mbox{ in } I\times\Omega,
\end{align}
where $\bb{f}$ is known and the resistivity $\rho$  is modeled by
\begin{equation}
                 \rho = \param | \nabla \times \bb u |^{p-2},                                                                                            
\end{equation}
for some positive material properties $\param$ and $p $ typically between $20$ and $100$. 
The model also includes initial conditions ${\bb u}(0,\cdot) = {\bb u}_0(\cdot)$ and boundary conditions.
Although the boundary conditions are often imposed indirectly by means of a global current constraint,
this work will focus on straightforward, but more restrictive, tangential boundary conditions
$$
    \bb n \times \bb u = \bb g, \quad \text{over } I \times \partial \Omega,
$$
where $\bb n$ is the exterior normal along the boundary. For consistency, 
the initial conditions ${\bb u}_0$ and the source term $\bb f$ must be divergence free. 
More general boundary conditions were studied by Miranda et al. \cite{MirRodSan08}. The precise assumptions 
leading to this model can be found in \cite{LafSirWan14} and a description of how this macroscopic model 
relates to microscopic models of superconductivity can be found in \cite{Cha00}.

There is an obvious analogy between the operator $\nabla \times ( |\nabla \times \bb u |^{p-2} \nabla \times {\bb u})$
of the model \eqref{basic_model} and the $p$-Laplacian, namely $\nabla \cdot ( |\nabla u |^{p-2} \nabla u )$. 
Researchers, Yin \cite{Yin99,Yin02}, as well as Miranda, Rodrigues and Santos \cite{MirRodSan08} 
have exploited this analogy in order to construct a well-posedness theory for the continuous
problem. The key parts of that theory is the observation
that the $p$-curl is monotone and the domain must have a smooth $C^{1,1}$ boundary. 
Formal convergence as $p \to \infty$ of the power-law model to
the Bean model has also been established in 2D \cite{BarPri00} and in 3D \cite{YinLiZou02}. Smoothness of the
boundary is an essential constraint coming from the harmonic analysis in $W^{1,p}$ spaces \cite{JerKen95,Mit04,SimSoh92}.

As far as we know, the theory of convergence of finite element approximation using N\'ed\'elec elements, within the same $W^{1,p}$ framework of Yin, has yet to be established. On the other hand, using an electric field formulation of the $p$-curl problem, Slodi\u{c}ka and Jan\'{i}kov\'{a} showed convergence results within $L^2$ spaces for backward Euler semi-discretizations and fully-discretizations using linear N\'ed\'elec elements in \cite{Slo06,JanSlo08,JanSlo10}. However, their work has only focused on a priori error estimates.

The main result of this paper, an a posteriori error estimate, appears to be the first residual-based error estimate for the problem \eqref{basic_model}.
In the work of Sirois et al. \cite{SiRoDu09}, an explicit adaptive time-stepping scheme was handled by SUNDIALS \cite{SUNDIALS} which contains sophisticated but generic error control strategies. The error estimates presented in this paper are residual based and resemble the a posteriori error estimators one finds for linear or linearized problems \cite{Ver98}. In fact, our results differ from those of Verf\"{u}rth in our treatment of the non-conformity of the approximation and in our circumvention of linearization.
Error estimation for FE approximate solutions of the $p$-Laplacian is quite well-developed and in fact, we mention the 
important work on reliable and efficient error estimation using quasi-norms \cite{LiuYan01,CarKlo03,CarLiuYan06a,DieKre08,BelDieKre12}. 
In recent work of El Alaoui et al. \cite{ElAErnVoh11}, quasi-norm error estimates 
were obtained by re-interpreting the estimators in terms of flux corrections
satisfying specific properties. It appears that their approach could be adapted to the $p$-curl using the tools we presented here to handle non-conformity issues. The error estimate presented here also controls the error in an important quantity of interest, the AC loss over one cycle. We have included a proof of the well-posedness for the straightforward semi-discretization often considered within the engineering community. Numerical results are presented to assess the quality of the error estimators. These experiments confirm the reliability of the error estimators on a class of moving front solutions in $2$D.

The novelty of this paper is the treatment of the lack of conformity of the N\'ed\'elec element approximations. Inspired largely by the
work of Carstensen, Ju and Orlando on the issue \cite{CarHuOrl07}, we have found that coercive estimates are sufficient
to obtain reliable error estimates. This is in stark contrast to most nonlinear problems which require a linearization
of the operator in a neighborhood of the numerical solution. Given that the semi-discretization considered here
is also found in commercial codes, and that the a posteriori error estimators of this paper are straightforward to implement, 
it appears that this work could be of interest to the engineering community.

The a posteriori error estimate also includes an interesting non-conformity error originating 
from the geometric defect between the approximation
of the $C^{1,1}$ domain $\Omega$, required for the continuous problem, by the polyhedral domain $\Omega_h$
required for the finite element formulation. Even with the use of curved elements approximating the boundary, such a geometric defect could not be eliminated. This difficulty, which appears to be specific to nonlinear harmonic analysis in $L^p$ spaces \cite{JerKen95}, is carefully analyzed and reduced to a boundary term on $\partial \Omega_h$ mesuring 
our inability to represent the discrete solution over a $C^{1,1}$ domain. Morevover, the
techniques used required that the polyhedral mesh $\Omega_h$ be strictly included 
inside the domain $\Omega$ of the continuous problem. The paper includes
a novel construction of a family of uniformly regular polyhedral domains strictly inside a $C^{1,1}$ domain, based
on the work of Delfour \cite{Del00}, Oudot, Rineau and Yvinec \cite{Oudot2005}, and Talmor \cite{Tal97}.

The paper is organized as follows.
The second section presents a brief review of the functional analysis required for the a posteriori error estimation.
In Section \ref{app:WP-semidisc}, for the sake of completeness we include a demonstration of the well-posedness
of our semi-disretization of the $p$-curl problem.
The fourth section contains the proof of the main theorem. It is later extended in Section \ref{sec:ACloss} to the control of
the AC loss. The last section describes numerical results obtained when comparing the error estimator
to the exact error for a class of moving front solutions using the method of manufactured solutions and as well as convergence results for a backward Euler discretization.
In Appendix \ref{app:non-homomBC}, we have extended the a posteriori error estimator to the 
case of non-homogeneous tangential boundary conditions, exploiting again properties unique to
the $p$-Laplacian and the $p$-curl problem.

\section{Preliminaries}

This section reviews the main functional spaces over which the $p$-curl problem is examined and 
it states the strong and weak forms of the problem. The triangulation of the domain is carefully
discussed since it involves a non-conformity issue important to the $p$-curl problem. 
A brief review of the finite element discretization of the $p$-curl is given.
This section concludes with a detailed presentation of 
the two main technical tools, namely the Helmholtz-Weyl decomposition over $L^p$ spaces
and the quasi-interpolation operator of Sch{\"o}berl \cite{Sch07}.

Let $d=2,3$ and $\Omega$ be a bounded Lipschitz domain in $\mathbb{R}^d$. 
Let $k$ be a nonnegative integer and for $s\geq 0$ denote its integer part as $[s]$. 
Throughout, we denote $q$ as the H\"{o}lder conjugate exponent of $p$ satisfying $1=1/p+1/q$. 
Recall the following well-known Sobolev spaces \cite{Ada03}.
\begin{align*}
\W{k} &= \{v\in \Ldom{p}: D^{\alpha}
v \in \Ldom{p}^d, |\alpha| \leq k\}\\
\W{s} &= \left\{v\in \W{[s]}: \sum_{\substack{\alpha \in \mathbb{N}^d \\ |\alpha|=[s] }}\norm{\frac{D^{\alpha}v(x)-D^{\alpha}v(y)}{|x-y|^{d/p+s-[s]}}}_{L^p(\Omega\times\Omega)}^p<\infty, \right\} \\
W_0^{s,p}(\Omega) &= \{v\in \W{s}: \gamma_0(v)=0\} \\
\W{-s} &= \left(W_0^{s,q}(\Omega)\right)'
\end{align*}
For the $p$-curl problem, we will see later that minimal regularity suggests that we consider the following spaces with $\Omega$ being a bounded $C^{1,1}$ domain; see \cite{Mit04,AmrSel13} for more details on their properties and equivalent norms.
\begin{align*}
\Wcurl &= \{\bb v\in \Ldom{p}^d: \curl\bb v\in\Ldom{p}^d\}\\
\Wzcurl &= \{\bb v\in \Wcurl: \gamma_t(\bb v)=0\}\\
\Wdiv &= \{\bb v\in \Ldom{p}^d: \div\bb v\in\Ldom{p}\}\\
\Wdivz &= \{\bb v\in \Wdiv: \div\bb v=0\}
\end{align*}
\[
\X = \Wzcurl\cap\Wdivz
\]
Above, $\gamma_0:\W{1} \rightarrow W^{1-1/p,p}(\partial \Omega)$ is the continuous boundary trace operator and $\gamma_t: \Wcurl \rightarrow (W^{1-1/p,p}(\partial \Omega)^d)'$, $\gamma_n: \Wdiv \rightarrow W^{1-1/p,p}(\partial \Omega)'$ are the continuous tangential and normal trace operators satisfying \cite[Corollary B.57 and B.58]{ErnGue04}:
\begin{align}
\big( \gamma_t(\bb v),\gamma_0(\bb w) \big)_{\partial \Omega} &=\int_\Omega \bb v \cdot\curl\bb w \, dV -\int_\Omega \bb w \cdot \curl\bb v \, dV, &\label{tangTrace}\\ && \hspace{-20mm} \forall \bb v \in \Wcurl,  \bb w\in {W^{1,q}(\Omega)}^d, \nonumber \\
\big( \gamma_n(\bb v),\gamma_0(w) \big)_{\partial \Omega} &=\int_\Omega \bb v \cdot \grad w \, dV+\int_\Omega \div \bb v w \, dV, &\label{normalTrace}\\ && \hspace{-20mm} \forall \bb v \in \Wdiv, w\in {W^{1,q}(\Omega)}. \nonumber
\end{align} 
For sufficiently smooth functions $\bb v$ and $w$, these trace operators are simply $\gamma_0(w)=w|_{\partial \Omega}$, $\gamma_t(\bb v)=\bb n \times \bb v|_{\partial \Omega}$ 
and $\gamma_n(\bb v)=\bb n \cdot \bb v|_{\partial \Omega}$. Later, we will need the stability bound below \cite{Ada03}.

\begin{lemma}  \label{traceThm}
Let $\Omega$ be a bounded domain with a Lipschitz boundary. If $\boldsymbol v\in W^{1,p}(\Omega)$, then the boundary trace operator $\gamma_0:W^{1,p}(\Omega)\rightarrow L^p(\partial \Omega)$ is a continuous linear operator, i.e. there exist a constant $C>0$ such that,
\begin{align}
\norm{\gamma_0(\boldsymbol v)}_{L^p(\partial \Omega)} \leq C\norm{\boldsymbol v}_{W^{1,p}(\Omega)}. \label{traceBound}
\end{align}
\end{lemma}

As is customary for $L^2$ spaces, we write $W^{k,2}(\Omega)$ as $H^k(\Omega)$ and similarly we write $W^{2}(\text{div};\Omega)$ and $W^{2}(\text{curl};\Omega)$ as $H(\text{div};\Omega)$ and $H(\text{curl};\Omega)$, respectively.

If $\bb u\in\Ldom{q}^d, \bb v\in\Ldom{p}^d$, we denote the pairing
\begin{align*}
\rpair{\bb u,\bb v} :=\int_\Omega \bb u\cdot\bb v \,dV,
\end{align*} 
and define the nonlinear operator $\mathcal{P}:\Wcurl\rightarrow \Wcurl'$,
\begin{align}
\pair{\mathcal{P}(\bb u),\bb v} :=\rpair{\rho(\curl \bb u)\curl \bb u,\curl \bb v}.
\end{align} 
Indeed, by Holder's inequality, these pairings are well-defined since,
\begin{align*}
\rpair{\bb u,\bb v} &\leq \norm{\bb u}_{\Ldom{q}}\norm{\bb v}_{\Ldom{p}}, \\
\pair{\mathcal{P}(\bb u),\bb v}  &\leq \alpha \norm{\curl\bb u}_{\Ldom{p}}^{p/q}\norm{\curl \bb v}_{\Ldom{p}}.
\end{align*}

Over the time interval $I=[0,T]$, the $p$-curl problem arising from applied superconductivity is the following nonlinear evolutionary equation:
\begin{equation}  \label{pcurl-complete}
\begin{array}{r@{\hspace{2pt}}c@{\hspace{2pt}}ll}
       \pdt\bb u+\curl \left[ \rho(\curl \bb u)\curl \bb u\right]&=&\bb f,  &\mbox{ in } I\times\Omega, \\
       \nabla \cdot \bb u &= &0,  & \mbox{ in } I\times\Omega,  \\
      \bb u(0,\cdot)&=&\bb u_0(\cdot),  & \mbox{ in } \Omega, \\
      \bb n\times \bb u&=&0,  &\mbox{ on } I\times\partial \Omega,
\end{array}
\end{equation}
where $p\geq2$, $\rho$ is the nonlinear resistivity modeled by an isotropic power law $\rho(\curl \bb u)=\param|\curl \bb u|^{p-2}$ and $\param=E_0/(\mu J_c^{p-1})> 0$ is a material dependent constant. Moreover, it is assumed that $\div \bb u_0=0$ and $\div \bb f=0$ for all $t\in I$ in a manner to be made precise later.

The weak formulation of the $p$-curl problem is: \\ \\
\emph{
Given $\Omega$ a bounded $C^{1,1}$ domain, $\bb u_0\in \Wdivz$ and $\bb f \in \L{2}(I;W^q(\text{div}^0;\Omega))$, find $\bb u \in \L{2}(I;\X)\cap H^1(I;\Ldom{q})$ satisfying  $\bb u(0,\cdot)=\bb u_0(\cdot)$ and }
\begin{align}
\rpair{\pdt\bb u,\bb v}+\pair{\mathcal{P}(\bb u),\bb v}=\rpair{\bb f,\bb v}, \quad \forall \bb v \in L^2(I;\X) . \label{WF}
\end{align}\\ 
The well-posedness of the weak problem was established in the work of Yin et al. \cite{Yin02,YinLiZou02}. The stability
of the solution is characterized by two inequalities from Lemma 3.2 of \cite{YinLiZou02}, one of which is 
given by
\begin{align}   \label{ineq:stability-continuous}
  \int_0^T \norm{\partial_t \bb u(s)}_{L^2(\Omega)}^2 ds
          & \, +\sup_{t\in [0,T]} \norm{\nabla\times \bb u(t)}^p_{L^p(\Omega)}  \\
          \leq & \, C \norm{\nabla\times \bb u_{0}}_{L^p(\Omega)}^p 
                  + C \int_0^T \norm{\bb f(s)}_{L^2(\Omega)}^2ds \,.  \nonumber
\end{align}
We demonstrate a similar bound for our approximate solution in Theorem \ref{dwfWP}.

\subsection{Approximating a $C^{1,1}$ domain}

Being restricted to a $C^{1,1}$ domain $\Omega$, in part due to the well-posedness of the $p$-curl problem, 
we observe that the domain of the polyhedral mesh $\Omega_h$ cannot be equal to $\Omega$, and therefore that the solution $\bb u$ to \eqref{WF} and any finite element approximation $\bb u_h$ cannot be defined over the same domain. When comparing $\bb u$ and $\bb u_h$, this introduces a geometric non-conformity that requires us to construct a polyhedral mesh $\Omega_h$ that approximates the $C^{1,1}$ domain $\Omega$ sufficiently well. The construction of the mesh will exploit the fact that $C^{1,1}$ domains are (nearly) those with the weakest regularity for which tubular neighborhoods can be defined. 
For the sake of simplicity, the description will be given only in $\mathbb{R}^3$ although the 
modifications to $\mathbb{R}^2$ should be obvious.

Let $\{\Tau_h\}_{h>0}$ be a collection of shape-regular triangularization of $\Omega$ where $\Tau_h:=\{K\subset \Omega: K \text{ a tetrahedron in } \mathbb{R}^3\}$ with $h$ being the largest diameter over all $K\in\Tau_h$. Denote $\displaystyle \Omega_h:=\bigcup_{K\in \Tau_h} \overline{K}$ as the polyhedral mesh with the obvious constraints that are required to ensure that the set of faces $\mathcal{F}(\Omega_h)$ and edges $\mathcal{E}(\Omega_h)$ of $\Omega_h$ are well-defined. Also for each $K \in \Tau_h$, denote $h_K$ as the diameter of $K$ and $\rho_K$ as the diameter of the largest inscribed sphere within $K$. By definition of shape-regularity of $\{\Tau_h\}_{h>0}$, $\displaystyle \max_{K\in \Tau_h}\frac{h_K}{\rho_K} \leq \sigma$ for all $h>0$. Moreover, for each face $F$ on $\partial K$, we also denote $h_F$ as the diameter of $F$ and $\rho_F$ to be the diameter of the largest inscribed circle within $F$.  The following lemma is obtained by combining the trace theorem of Lemma \ref{traceThm} with a standard scaling argument. 
\begin{lemma} \label{discTraceIneq}
Let $K\in \Tau_h$ and $F$ be any face on $\partial K$. If $v\in W^{1,p}(K)$, then there exists a constant $C>0$ independent of $K$ and $v$ such that, 
\begin{align}
h_F^{1-p}\norm{\gamma_0(v)}_{L^p(F)}^p \leq C \big(h_F^{-p}\norm{v}_{L^p(K)}^p + \norm{\nabla v}_{L^p(K)}^p \big).  \label{scaledtraceThm}
\end{align} 
\end{lemma} 

Due to the geometric non-conformity, we will further be interested in a special class of triangulation of $\Omega$. For a bounded domain $\Omega \subset \mathbb{R}^3$, we define an {\it interior mesh} $\Tau_h$ to be a triangulation
of the domain $\Omega$ for which the union of all tetrahedron $\Omega_h$ is strictly contained in $\Omega$.  
If $\Omega$ is a convex $C^{1,1}$ domain and the vertices of $\partial \Omega_h$ lies 
within $\Omega$, then clearly $\Omega_h$ is an interior mesh of $\Omega$. 
For a fixed nonconvex bounded $C^{1,1}$ domain $\Omega$, the existence of a sequence
of triangulations for which the volume of the defect $\Omega\setminus \Omega_h$ vanishes uniformly, in some sense, is far from obvious. We begin with a fundamental result of Delfour \cite{Del00}, citing Lemma 2.1 from \cite{Del09}.

\begin{theorem} \label{thm:tubular-nhood}
Let $\Omega$ be a bounded domain with a non-empty $C^{1,1}$ boundary 
$\partial \Omega$.  There exists a number $d = d(\Omega) \in \mathbb{R}^+$, an open
neighborhood $U_{d}$ of $\partial \Omega$, and a bi-Lipschitzian map 
\begin{align*}
    \Gamma : 
    \partial \Omega \times [-d,d]  \longrightarrow \overline{U}_{d} \, ,
\end{align*}
satisfying
\begin{itemize}
\item[i)] the map $\Gamma(\cdot,0)$ is the identity over $\partial \Omega$;
\item[ii)] for each $s \in [-d,d] $ the image of the map 
     $\Gamma(\cdot, s) : \partial \Omega \longrightarrow \overline{U}_{d}$ 
      is a $C^{1,1}$ hypersurface;  
\item[iii)] for each fixed $x \in \partial \Omega$, the derivative $d\Gamma/ds (x,0) $
             is the exterior normal to the boundary at $x$;
\item[iv)] for all $(x,s) \in \partial \Omega \times [-d,0)$, the image
             $\Gamma(x,s) $ is inside $\Omega$.                           
\end{itemize}
\end{theorem}

For domains with weak $C^{1,1}$ regularity, there exists a triangulation algorithm 
developed by Oudot, Rineau and 
Yvinec \cite{Oudot2005} which constructs  a mesh  arbitrarily close to the boundary.
The algorithm only requires an oracle that (i) determines if a point is inside the
domain, and (ii) computes the intersection point between the boundary and a segment 
in generic position. This algorithm has been implemented in CGAL \cite{CGAL}
and distinguishes itself from conventional algorithms that are usually restricted to polyhedral domains. 
We present here a form of their result specifically adapted to our situation.

\begin{theorem} \label{thm:triangulation}
Let $\Omega$ be a bounded domain with a non-empty $C^{1,1}$ boundary $\partial \Omega$.
There exists a positive constant $\sigma_m$ and a Lipschitz sizing field  on $\partial \Omega$,
$$
        r_{\partial \Omega} : \partial \Omega \longrightarrow \mathbb{R} \, ,
$$
such that for every $\delta \in (0,d(\Omega))$, there exists 
a triangulation $\Tau_h$ of an interior mesh $\Omega_h $ of $ \Omega$  satisfying
\begin{itemize}
\item[i)] $\partial \Omega_h \subset \Gamma(\partial \Omega, [-\delta,0) )$;      
\item[ii)] for all faces $F$ along the boundary of the mesh, 
             \begin{equation}  \label{ineq:constraint_boundary}
                \frac{h_F}{\delta} \leq r_{\partial \Omega}(x) \, , \quad \forall \, x \in F ;
             \end{equation}                                                    
\item[iii)] the triangulation is shape-regular, that is
               $$
                       \frac{h_K}{\rho_K} \leq \sigma_m \, , \quad \forall \, K \in \Tau_h .
               $$                  
\end{itemize}
\end{theorem}

\begin{proof} 
For every positive value of $\delta $ less than $d(\Omega)$, define the $C^{1,1}$ domain
$$
        \Omega^{(\delta)}:= \Omega \setminus \Gamma\big(\partial \Omega,[-\delta/2,0] \big) \, .
$$        
We will now construct a triangulation of $\Omega^{(\delta)}$ that  guarantees that the union of the 
tetrahedrons of the mesh $\Omega_h$ satisfy 
\begin{equation}     \label{boundary_interior_mesh}
          \partial \Omega_h \subset \Gamma\big(\partial \Omega,[-3\delta/4,-\delta/4] \big) \, .
\end{equation}
The algorithm of Oudot, Rineau and Yvinec allows us to construct a triangulation of a $C^{1,1}$ domain,
but not necessarily produce an interior mesh. The iterative algorithm begins 
by choosing a value for the bound $B$ on the radius-edge ratio, that is the ratio
\begin{equation}  \label{radius-edge-ratio}
         \frac{h_K}{\gamma_K} \leq B \, , \qquad \forall K \in \Tau_h \, ,
\end{equation}
where $\gamma_K$ is the length of the shortest edge of $K$. Points
are then randomly selected inside $\Omega^{(\delta)}$ and on the 
boundary $\partial \Omega^{(\delta)}$. Tetrahedrons and
faces on the boundary are selectively refined by inserting the circumcenter and connecting
the vertices to the circumcenter until both \eqref{ineq:constraint_boundary} 
and \eqref{radius-edge-ratio} are satisfied.

We will show that if a face $F$ on the boundary of the mesh satisfies a constraint
$ h_F \leq C \delta$, then 
\begin{equation} \label{mesh_face_condition}
       F \subset \Gamma(\partial \Omega, [-3\delta/4,-\delta/4]) \, .
\end{equation}
Choose a face $F$ belonging to the boundary of $\Omega_h$, 
suppose one of its three vertices is $x_1 \in \partial \Omega^{(\delta)}$. 
For a point $x \in F$, define $D = \| x_1 - x \|$ and the smooth function 
$g(\eta) = (1-\eta/D)x_1 + \eta /D x $ describing, for arc length $\eta \in [0,D]$,
the straight line segment connecting $x_1$ to $x$.
If $P(x,s) = s$ is the projection onto the second variable, then
the Lipschitz continuity of the inverse of $\Gamma$ implies that there exists a
constant $M$ such that
$$
  \Big| P \circ \Gamma^{-1} \circ g( D ) 
          - P \circ \Gamma^{-1} \circ g( 0 ) \Big|
          \leq M D \leq M  h_F \, .
$$
Therefore, if all the vertices belong to $\partial \Omega^{(\delta)}$ and if the face $F$ satisfies 
\begin{equation}  \label{relation_h_delta}
                M h_F \leq M h \leq \frac{\delta}{4} \, \quad \Longleftrightarrow \quad h \leq  C \delta \, ,
\end{equation}
for some fixed $C$, then the condition \eqref{mesh_face_condition} holds
and the mesh $\Omega_h$ is strictly inside $\Omega$.
Moreover, the constant $C$ depends only on the Lipschitz constant of the
boundary $\partial \Omega$, and not on $\delta$. We remark that these
observations allow us to assign to each vertex the value $(4M)^{-1}$, which
depends only on $\partial \Omega^{(\delta)}$, and then construct 
the sizing field as a piecewise linear interpolant of these values.
The inverse of $\Gamma$ then allows the sizing field over $\partial \Omega^{(\delta)}$ 
to be defined over $\partial \Omega$.


Finally, we address the shape-regularity of the mesh. In fact, the algorithm 
by Oudot et al. only produces meshes with bounded radius-edge ratios \eqref{radius-edge-ratio}
and these meshes may contain so-called {\it slivers}, that is tetrahedrons
possessing one vertex close to the plane of the three others vertices
yet with angles bounded from below. There exists very
efficient algorithms to remove such slivers, but in fact the Sliver Theorem 
of Talmor states that if a mesh satisfies the radius-edge ratio condition, then
there exists a topologically equivalent mesh that is shape-regular \cite{Tal97}.
From a mathematical perspective, the shape-regular condition $\sigma_m$ 
therefore follows from the choice of the constraint $B$.
\end{proof}

The main motivation for introducing an interior mesh is the following simple extension result.


\begin{lemma} \label{lem:exten}
Let $\Omega_h$ be an interior mesh of $\Omega$. 
For each $\bb v\in W^p_0(\text{curl};\Omega_h)$, its trivial extension by zero defined by
\begin{equation*}
\tilde{\bb v}(x) := 
\begin{cases}
\bb v, & x\in \Omega_h \, , \\
\bb 0, &x\in \Omega\setminus\Omega_h \,,
\end{cases} 
\end{equation*}belongs to $W^p_0(\text{curl}; \Omega)$.
\end{lemma}

\begin{proof} Clearly, $\tilde{\bb v}\in L^p(\Omega)$. Since $\left.\tilde{\bb v}\right|_{\Omega_h} = \bb v \in W^p_0(\text{curl};\Omega_h)$ and $\left. \tilde{\bb v}\right|_{\Omega\setminus\Omega_h} = \bb 0 \in W^p_0(\text{curl};\Omega\setminus\Omega_h)$, the tangential jump $[[\gamma_t(\tilde{\bb v})]]_{\partial \Omega_h}= \bb 0$. So, it follows from \eqref{tangTrace} that, $\tilde{\bb v}\in W^p(\text{curl}; \Omega)$ and clearly $\gamma_t(\tilde{\bb v})|_{\partial \Omega}=\bb 0$.
\end{proof}

\subsection{Semi-discretization of $p$-curl problem by N\'{e}d\'{e}lec finite elements}

In $\mathbb{R}^3$, the $k$-th order N\'{e}d\'{e}lec finite element space of the first kind \cite{Ned80} and with zero tangential trace can be defined as,
\begin{align}
\Xh^{(k)}&:=\{\bb v\in W^p(\text{curl};\Omega_h):\bb v|_K=\bb a (\bb x)+\bb b (\bb x)\times \bb x, \bb a,\bb b \in[\mathbb{P}_{k-1}]^3, K\in \Tau_h\}, \\
\Xhz^{(k)}&:=\Xh^{(k)}\cap W_0^p(\text{curl};\Omega_h),
\end{align} 
where $(\mathbb{P}_k)^3$ is the space of vector fields with polynomial components of at most degree $k$. 
Recall that the finite element space $\Xhz^{(k)}$ is uniquely determined by identifying the degrees of freedom of the surface integral along faces and edges between any two neighboring elements. Since an element-wise $W^p(\text{curl};K)$ defined function that is continuous tangentially along faces and edges is a global $\Wcurlh$ function, $\Xhz^{(k)}\subset \Wzcurlh$. Moreover $\Xhz^{(1)}$ is known to be locally divergence-free, i.e. $\div\bb v|_K=0$ for $\bb v\in\Xhz^{(1)}$, and thus it is an element-wise $W^p(\text{div}^0;K)$ defined function. Unfortunately, higher order elements
will not be in $W^p(\text{div}^0;K)$. In any case, $\Xhz^{(k)}$ can be discontinuous in the normal direction to faces and edges and hence in general is not a global $W^p(\text{div};\Omega_h)$ function. In particular, $\Xhz^{(k)}\not\subset \X$.


This leads us to the non-conforming semi-discrete weak formulation of the $p$-curl problem:\\ \\
\emph{
Given $\bb u_{0,h}\in \Xhz^{(k)}$ and $\bb f \in C(I;W^q(\text{div}^0;\Omega))$, find $\bb u_h\in C^1(I;\Xhz^{(k)})$
satisfying $\bb u_h(0,\cdot)=\bb u_{0,h}(\cdot)$ and}
\begin{align}
\rpairh{\pdt\bb {u_h},\bb v_h}+\pairh{\mathcal{P}(\bb u_h),\bb v_h}=\rpairh{\bb f,\bb v_h} , \quad \forall \bb v_h \in \Xhz^{(k)} .\label{DWF}
\end{align}
Due to the non-conformity, well-posedness of the semi-discretization does not necessarily follow from the well-posedness of the weak formulation. By a local existence argument and a priori estimate, we show that the semi-discretization is well-posed in Section \ref{app:WP-semidisc}. Note that, while the weak formulation only requires $\bb f$ to be $L^2$ in $t$, we need $\bb f$ to be continuous in $t$ in order apply Picard's local existence theorem.



\subsection{Helmholtz-Weyl decomposition of $W^p_0(\text{curl};\Omega)$ functions}

We now proceed with a rather detailed review of the Helmholtz-Weyl decomposition for $\L{p}$ spaces. 
This is needed to address the non-conformity in a manner similar to the work of \cite{CarHuOrl07}.
The most technical aspects concerning the $p$-curl problem turn out to be related to this decomposition, not
only because of the Banach nature of the $L^p$ spaces concerned, but also because it imposes strict limits on the
regularity of the boundary.

Define $L_\sigma^p(\Omega):=$ closure of $\{\bb v \in C_0^\infty(\Omega)^d: \div \bb v=0\}$ with respect to $\L{p}\text{ norm}$. 
A standard formulation of the decomposition is the following. \\ \\
\emph{
There exists a positive constant $C=C(\Omega,p,d)$ such that for any $\bb v\in \Ldom{p}^d$, there exists $\phi \in \W{1}/\mathbb{R}$ and $\bb z \in L_\sigma^p(\Omega)$ 
for which $\bb v =\bb z+\grad\phi$ and 
\begin{align}
\norm{\bb z}_{\L{p}(\Omega)} + \norm{\grad \phi}_{\L{p}(\Omega)}\leq C\norm{\bb v}_{\L{p}(\Omega)}. \label{HWdecomp}
\end{align}}
\\
When the vector field  has zero boundary trace, then the Helmholtz-Weyl decomposition is as follows.\\ \\
\emph{
There exists a positive constant $C=C(\Omega,p,d)$ such that for any $\bb v\in \Ldom{p}^d$, there exists $\phi \in \Wz$ and $\bb z \in \Wdivz$ for which 
$\bb v =\bb z+\grad\phi$ and
\begin{align}
\norm{\bb z}_{\L{p}(\Omega)} + \norm{\grad \phi}_{\L{p}(\Omega)}\leq C\norm{\bb v}_{\L{p}(\Omega)}. \label{HWzdecomp}
\end{align}}
\\
While the decomposition when $p=2$ can be studied using tools no more complicated than the Lax-Milgram theorem, the case for general $p$ is much more subtle.
It has been observed (for example \cite[Lemma III 1.2]{Gal11}) that the existence of the Helmholtz-Weyl decomposition of \eqref{HWdecomp} 
is equivalent to the solvability of the following Neumann problem over $\Omega$.\\ \\
\emph{Given $\bb v\in \Ldom{p}^d$, find $\phi \in \W{1}/\mathbb{R}$ such that for all $\psi \in W^{1,q}(\Omega)/\mathbb{R}$,
\begin{align*}
\rpair{\grad\phi,\grad\psi}=\rpair{\bb v,\grad \psi}.
\end{align*}}Similarly, the existence of Helmholtz-Weyl decomposition of \eqref{HWzdecomp} is equivalent to the solvability of the Dirichlet problem below.\\ \\
\emph{Given $\bb v\in \Ldom{p}^d$, find $\phi \in \Wz$ such that for all $\psi \in W_0^{1,q}(\Omega)$,
\begin{align*}
\rpair{\grad\phi,\grad\psi}=\rpair{\bb v,\grad \psi}.
\end{align*}}In particular, if $\Omega\subset\mathbb{R}^d$ is a bounded Lipschitz domain then for some $\epsilon(\Omega)>0$ depending on the Lipschitz constant of $\Omega$, it was shown in \cite{FabMenMit98} that the above Neumann problem has a solution in a sharp region near $p\in(3/2-\epsilon,3+\epsilon)$. Similarly, \cite{JerKen95} showed that the above Dirichlet problem has a solution in a sharp region near $p\in(2/(1+\epsilon),2/(1-\epsilon))$. This implies the Helmholtz-Weyl decomposition does not hold in general for bounded Lipschitz domains, which is unfortunate since such domains do arise in engineering applications of superconductors. Thus, we are forced to restrict to bounded $C^{1,1}$ domains, which are consistent with the regularity of the boundary required for the well-posedness of the $p$-curl problem given by \cite{YinLiZou02}.

The Helmholtz-Weyl decomposition for $L^2$ was first demonstrated by \cite{Wey40} and for $L^p$ by \cite{FujMor77} for smooth bounded domains. For $1<p<\infty$, to the best of our knowledge, the weakest regularity requirement for the Helmholtz-Weyl decomposition to hold are bounded $C^1$ domains \cite{SimSoh92,SimSoh96} and more recently for bounded convex Lipschitz domains \cite{GenShe10}. 
\begin{theorem}\cite[Theorem II.1.1]{SimSoh96} 
Let $\Omega\subset \mathbb{R}^d$ be bounded $C^1$ domain and let $1<p<\infty$. Then the Helmholtz-Weyl decomposition \eqref{HWzdecomp} holds. \label{C1decomp}
\end{theorem}

\begin{theorem}\cite[Theorem 1.3]{GenShe10} 
Let $\Omega\subset \mathbb{R}^d$ be a bounded convex Lipschitz domain and let $1<p<\infty$. Then the Helmholtz-Weyl decomposition \eqref{HWdecomp} holds. \label{convexDecomp}
\end{theorem}
We also mention that Amrouche et al. $\cite{AmrSel13}$ have published an $\L{p}$ version of the Hodge decomposition for domains with $C^{1,1}$ boundary. We now use Theorem \ref{C1decomp} to derive a new Helmholtz-Weyl decomposition for $\Wzcurl$ for bounded $C^1$ domain.

\begin{lemma} \label{Wzcurldecomp}
Let $\Omega\subset \mathbb{R}^d$ be a bounded simply connected $C^1$ domain and let $2\leq p<\infty$. Then the following direct sum holds,
\begin{align*}
\Wzcurl=\X\oplus\nabla \Wz.
\end{align*}
In other words, for any $\bb v\in \Wzcurl$, there exists unique $\phi \in \Wz$ and $\bb z\in \X$ such that $\bb v = \bb z+\grad \phi$ satisfying,
\begin{align}
\norm{\bb z}_{\L{p}(\Omega)} + \norm{\grad \phi}_{\L{p}(\Omega)}\leq C\norm{\bb v}_{\L{p}(\Omega)}, \quad C=C(\Omega,p,d)>0.
\end{align}
\end{lemma}
\begin{proof}
Let $\bb v\in \Wzcurl\subset \Ldom{p}^d$. Then by Theorem \ref{C1decomp}, $\bb v=\grad\phi+\bb z$ for some $\phi\in \Wz$ and $\bb z\in \Wdivz$. Since $\grad \Wz\subset\Wcurl$, $\gamma_t(\grad \phi)$ is well defined. Let $\{\phi_k\in C_0^\infty(\Omega)\}$ converging to $\phi$ in $\Wz$. Since $\gamma_0(\grad \phi_k)=0$ and so $\gamma_t(\grad \phi_k)=0$, then by continuity of the tangential trace operator $\gamma_t(\grad \phi)=0$ and so $\bb z = \bb v-\grad\phi\in \Wzcurl$. I.e. $\bb z\in \X$.

To show the sum is direct, suppose $\bb v\in \X\cap\grad\Wz$. Then $\bb v=\grad \phi$ for some $\phi\in \Wz$. Since $\bb v\in \X$, for all $\psi \in W_0^{1,q}(\Omega)$,
\begin{align}
0= \rpair{\bb v,\grad \psi}=\rpair{\grad \phi,\grad \psi}\label{p1e1}
\end{align} As $p\geq 2\geq q>1$, $\phi\in W_0^{1,2}(\Omega) \subset W_0^{1,q}(\Omega)$. Setting $\psi=\phi$ in \eqref{p1e1} implies $\norm{\grad \phi}_{\Ldom{2}}=0$ and hence $\phi=0$ a.e. by Friedrichs' inequality. I.e. $\bb v=\grad \phi=0$.
\end{proof}

Finally, we conclude with the quasi-interpolation operator $\Pi_h$ of Sch{\"o}berl \cite[Theorem 1]{Sch07}, which for N\'ed\'elec elements plays the same role  the Cl\'ement operator does for Lagrange elements.
\begin{theorem} \label{quasiHyp}
Consider a bounded polyhedral domain $\Omega_h \subset \mathbb{R}^3$ possessing a
triangulation $\Tau_h$.
There exists a quasi-interpolation operator $\Pi_h:H(\text{curl};\Omega_h)\rightarrow \Xh^{(k)}$ with the property: 
for any $\bb v\in H(\text{curl};\Omega_h)$, there exists $\phi \in H^1(\Omega_h)$ and $\bb w\in H^1(\Omega_h)^3$ such that,
\begin{align}
\bb v-\Pi_h \bb v = \grad \phi + \bb w \, . \label{decompEqn}
\end{align} 
Moreover,  on each $K\in \Tau_h$ there exists an element patch $\omega_K$ of $\bar{K}$ and 
a constant $C>0$ depending only on the shape constants of the elements in $\omega_K$ such 
that $\phi, \boldsymbol w$ satisfy
\begin{align}
h_K^{-1}\norm{\phi}_{L^2(K)}+\norm{\nabla \phi}_{L^2(K)}&\leq C\norm{\boldsymbol v}_{L^2(\omega_K)}, \label{quasiE1} \\
h_K^{-1}\norm{\boldsymbol w}_{L^2(K)}+\norm{\nabla \boldsymbol w}_{L^2(K)}&\leq C\norm{\nabla \times\boldsymbol v}_{L^2(\omega_K)}. \label{quasiE2}
\end{align}
\end{theorem} 

\section{Well-posedness of the semi-discretization}
\label{app:WP-semidisc}

This section contains a short proof of the well-posedness of the semi-discrete weak formulation of \eqref{DWF}. 
The well-posedness is not required for the construction of the a posteriori error estimators
in the following section, and so this section can be read independently of the others. 
Nevertheless, for the sake of accessibility, this topic is best discussed first.

\begin{theorem} \label{dwfWP}
There exists a unique solution $\bb u_h \in C^1(I;\Xhz^{(k)})$ satisfying the semi-discrete weak formulation of \eqref{DWF}. Moreover, the stability estimates hold,
\begin{align}
\sup_{t\in [0,T]} \norm{\bb u_h(t)}_{L^2(\Omega_h)}^2+2\int_0^T \norm{\nabla\times \bb u_h(s)}^p_{L^p(\Omega_h)} ds \nonumber\\
&\hspace{-30mm} \leq e\left(\norm{\bb u_{0,h}}_{L^2(\Omega_h)}^2 + T\int_0^T \norm{\bb f(s)}_{L^2(\Omega_h)}^2ds\right), \label{stabEst1}\\
\int_0^T \norm{\partial_t \bb u_h(s)}_{L^2(\Omega_h)}^2 ds+\sup_{t\in [0,T]} \norm{\nabla\times \bb u_h(t)}^p_{L^p(\Omega_h)} \nonumber \\
&\hspace{-30mm} \leq \norm{\nabla\times \bb u_{0,h}}_{L^p(\Omega_h)}^p + \frac{p^2}{4(p-1)}\int_0^T \norm{\bb f(s)}_{L^2(\Omega_h)}^2ds. \label{stabEst2}
\end{align}
\end{theorem}
\begin{proof}
The space of $k$-th order N\'ed\'elec elements  $\Xhz^{(k)}$ is a closed subspace of $\Wcurlh$ and we restrict
the norm of $\Wcurlh$ to it, 
$$
\norm{\bb v_h}_{\Wcurlh}^p = \norm{\bb v_h}_{L^p(\Omega_h)}^p + \norm{\nabla\times\bb v_h}_{L^p(\Omega_h)}^p, \hspace{4mm} \bb v_h\in \Xhz^{(k)}.
$$
By Riesz representation theorem for $L^p$ functions, there is an isometry $\Phi:L^q(\Omega_h) \rightarrow L^p(\Omega_h)'$,
also known as the Riesz map. Then we can view the semi-discrete weak formulation of \eqref{DWF} as seeking an unique solution $\bb u_h\in C^1(I;\Xhz^{(k)})$ to the first order ODEs, 
\begin{equation}
\Phi\circ\partial_t \bb u_h(t) = - \mathcal{P}(\bb u_h(t))+ \bb f(t). \label{DWFODE}
\end{equation}
The proof proceeds in 2 steps. First, we show local existence for \eqref{DWFODE}. Second, we extend its interval of existence to $I$ by a priori estimates.

To show local existence, we verify that the right hand side of \eqref{DWFODE} is continuous in $t$ and locally Lipschitz continuous in $\bb u_h$. Indeed, since $\bb f\in C(I;W^q(\text{div}^0;\Omega))$ with $q<2$ and $\Omega_h\subset \Omega$, $\bb f\in L^q(\Omega_h)$ for all $t\in I$. This implies for any $\bb v \in \Wcurlh$ and $t,s\in I$,
\begin{align*}
|\rpairh{\bb f(t)-\bb f(s), \bb v}| &\leq  \norm{\bb f(t)-\bb f(s)}_{L^q(\Omega_h)} \norm{\bb v}_{L^p(\Omega_h)}\\
&\leq \norm{\bb f(t)-\bb f(s)}_{L^q(\Omega_h)} \norm{\bb v}_{\Wcurlh}.
\end{align*} It follows that,
\begin{align*}
\norm{\bb f(t)-\bb f(s)}_{{\Wcurlh}'} &:= \sup_{0\neq \bb v \in \Wcurlh}\frac{|\rpairh{\bb f(t)-\bb f(s), \bb v}|}{\norm{\bb v}_{\Wcurlh}} \\ &\leq \norm{\bb f(t)-\bb f(s)}_{L^q(\Omega_h)},
\end{align*} which tends to $0$ as $s\rightarrow t$. This shows $\bb f(t)\in {\Wcurlh}'$ is continuous in $t$.

Now recall from \cite[Lemma 2.2]{BarLiu94}, that the following equality holds for some $C_p>0$,
\begin{equation*}
\left||\bb x|^{p-2}\bb x-|\bb y|^{p-2}\bb y\right|\leq C_p|\bb x-\bb y| (|\bb x| + |\bb y|)^{p-2}, \quad \forall \bb x,\bb y \in \mathbb{R}^d.
\end{equation*} So for any $\bb u, \bb v, \bb w \in \Wcurlh$, it follows from the above inequality and H\"older's inequality with $p>2$ so that $r:=\frac{p}{q}>1$, $s:=\frac{p}{q(p-2)}>1$ and $\frac{1}{r}+\frac{1}{s}=1$,
\begin{align*}
|\pairh{\mathcal{P}(\bb u)-\mathcal{P}(\bb w), \bb v}| &\leq \int_{\Omega_h} 
          \left||\nabla \times \bb u|^{p-2}\nabla \times \bb u - |\nabla \times \bb w|^{p-2}\nabla \times \bb w \right| |\nabla \times\bb v| \, dV \\
&\hspace{-20mm}\leq C_p \int_{\Omega_h} |\nabla \times (\bb u -\bb w)|(|\nabla \times\bb u|+|\nabla \times\bb w|)^{p-2}|\nabla \times\bb v| \, dV\\
&\hspace{-20mm}\leq C_p \norm{\nabla \times \bb v}_{L^p(\Omega_h)}\left(\int_{\Omega_h}|\nabla \times (\bb u -\bb w)|^q(|\nabla \times\bb u|+|\nabla \times\bb w|)^{(p-2)q}dV\right)^{1/q}\\
&\hspace{-20mm}\leq C_p \norm{\nabla \times \bb v}_{L^p(\Omega_h)}\left(\int_{\Omega_h}|\nabla \times (\bb u -\bb w)|^{qr} dV\right)^{1/qr} \times\\
&\hspace{+10mm} \left(\int_{\Omega_h}(|\nabla \times\bb u|+|\nabla \times\bb w|)^{(p-2)qs}dV\right)^{1/qs}\\
&\hspace{-20mm}= C_p\norm{\nabla \times\bb v}_{L^p(\Omega_h)}\norm{\nabla \times(\bb u-\bb w)}_{L^p(\Omega_h)}\norm{|\nabla \times\bb u|+|\nabla \times\bb w|}_{L^p(\Omega_h)}^{p-2} \\
&\hspace{-20mm}\leq C_p\norm{\bb v}_{\Wcurlh}\norm{\bb u-\bb w}_{\Wcurlh}\norm{|\nabla \times\bb u|+|\nabla \times\bb w|}_{L^p(\Omega_h)}^{p-2}.
\end{align*} Moreover, the case for $p=2$ follows directly from Cauchy-Schwarz inequality. Thus, we have that for any compact subset $A\subset \Xhz^{(k)}$ and any $\bb u_h, \bb w_h\in A$,
\begin{align*}
\norm{\mathcal{P}(\bb u_h)-\mathcal{P}(\bb w_h)}_{{\Wcurlh}'} &:= \sup_{0\neq \bb v\in \Wcurlh}\frac{|\pairh{\mathcal{P}(\bb u_h)-\mathcal{P}(\bb w_h), \bb v}|}{\norm{\bb v}_{\Wcurlh}} \\
&\hspace{-30mm}\leq \left(C_p \max_{\bb y_h, \bb z_h\in A} \norm{|\nabla \times\bb y_h|+|\nabla \times\bb z_h|}_{L^p(\Omega_h)}^{p-2}\right) \norm{\bb u_h-\bb w_h}_{\Wcurlh}.
\end{align*} 
This shows that $\mathcal{P}(\bb u_h)$ is locally Lipschitz continuous in $\bb u_h$. Thus, by Picard's existence theorem, there exists an unique local solution $\bb u_h \in C^1([0,\tilde{T});\Xhz^{(k)})$ to \eqref{DWF}, with $[0,\tilde{T}) \subset I$. 

Finally, we extend $[0,\tilde{T})$ to $I$ by showing the following a priori estimates. At every $t\in [0,\tilde{T})$, we have $\bb u_h(t,\cdot) \in \Xhz^{(k)}$. Setting now $\bb v_h = \bb u_h$ in \eqref{DWF}, and combining with Young's inequality and Gronwall's inequality implies
\begin{align*}
&\frac{d}{dt}\norm{\bb u_h}_{L^2(\Omega_h)}^2 + 2\norm{\nabla \times \bb u_h}_{L^p(\Omega_h)}^p \leq \epsilon \norm{\bb u_h}_{L^2(\Omega_h)}^2+ \frac{1}{\epsilon}\norm{\bb f}_{L^2(\Omega_h)}^2 \\
\Rightarrow &\norm{\bb u_h(t)}_{L^2(\Omega_h)}^2 + 2\int_0^t \norm{\nabla \times \bb u_h(s)}_{L^p(\Omega_h)}^p ds \\&
\hskip 40mm\leq e^{\epsilon T}\norm{\bb u_{0,h}}_{L^2(\Omega_h)}^2 +\frac{e^{\epsilon T}}{\epsilon}\int_0^T\norm{\bb f}_{L^2(\Omega_h)}^2 ds
\end{align*} Thus, taking supremum on the left hand side and setting $\epsilon = \frac{1}{T}$ shows  the stability estimate \eqref{stabEst1}, which implies $[0,\tilde{T})$ can be extended to $I$. Similarly, the second stability estimate \eqref{stabEst2} follows by setting $\bb v_h = \partial_t\bb u_h$ in \eqref{DWF} and noting that $\frac{1}{p}\frac{d}{dt} \norm{\nabla \times \bb u_h}_{L^p(\Omega_h)}^p = \pairh{\mathcal{P}(\bb u_h), \partial_t\bb u_h}$. 
\end{proof}

\section{A posteriori error estimator}
\label{sec:apee}

This section contains the main result of this paper, Theorem \ref{thm:main_thm2}. The proof follows the usual residual-based approach except for the
treatment of the non-conformity and nonlinearity. 
We begin with Lemma \ref{ncTest} which enables us to test the weak formulation with a larger test space. This is then used to
bound the error, as stated in Theorem \ref{thm:main_thm1}. Afterwards, stability estimates for both the trace operator and Sch\"oberl's quasi-interpolation operator allow us to combine the local estimate into a global estimate of Theorem \ref{thm:main_thm2}.

\begin{lemma} \label{ncTest}
Consider a $C^1$ simply connected bounded domain $\Omega$ and a source term $\bb f\in L^2(I;W^q(\text{div}^0;\Omega))$. 
Assume that $\bb u$ is a weak solution to \eqref{WF}, then 
\begin{align}
\rpair{\pdt\bb u,\bb v}+\pair{\mathcal{P}(\bb u),\bb v}=\rpair{\bb f,\bb v}, \quad \forall \bb v\in \Wzcurl.
\end{align}
\end{lemma}
\begin{proof} Let $\bb v\in \Wzcurl$. By Lemma \ref{Wzcurldecomp}, $\bb v=\bb z+\grad \phi$ for some $\phi \in \Wz$ and $\bb z\in \X$. Since $\bb u\in \X \subset\Wdivz$, $\bb f\in W^q(\text{div}^0;\Omega)$ and $\curl\grad \phi=0$ is well-defined for $\phi \in \Wz$,
\begin{align*}
\rpair{\pdt\bb u,\bb v}+\pair{\mathcal{P}(\bb u),\bb v} &= \left[\frac{}{}\rpair{\pdt\bb u,\bb z}+\pair{\mathcal{P}(\bb u),\bb z}\right] + \rpair{\pdt\bb u,\grad \phi}+\pair{\mathcal{P}(\bb u),\grad \phi}\\
&=\rpair{\bb f,\bb z}+\frac{d}{dt}\underbrace{\rpair{\bb u,\grad \phi}}_{=0}+ \rpair{\rho(\curl\bb u)\curl\bb u,\curl\grad \phi}\\
&=\rpair{\bb f,\bb z}+\underbrace{\rpair{\bb f,\grad \phi}}_{=0}\\
&=\rpair{\bb f,\bb v}.
\end{align*}
We remark that the interchange of differentiation and integration was permitted by Theorem (2.27) of \cite{Fol84}.
\end{proof}

Due to the discrepancy of the tangential boundary condition between $\bb u$ and $\bb u_h$ on $\Omega_h$, we will also need to decompose $\Xh^{(k)}$ into two contributions which are associated with the interior and boundary elements of $\Omega_h$. Specifically for $\Omega_h\subset\mathbb{R}^3$, $\bb v_h\in \Xh^{(k)}$ can be expressed as linear combinations of global shape functions $\{\bb \psi_{E,i}\}\cup \{\bb \psi_{F,i}\} \cup \{\bb \psi_{K,i}\}$ by assigning the same degrees of freedom along tangential components of $\bb v_h$ on common edges and faces \cite{Mon03,ErnGue04},
\begin{align*}
\bb v_h &= \sum_{\substack{E\in \mathcal{E}(\Omega_h)\\1\leq i \leq N_e}}\left(\int_E \bb v_h\cdot \bb\tau p_i ds\right)\bb \psi_{E,i}+\sum_{\substack{F\in \mathcal{F}(\Omega_h)\\ 1\leq i\leq N_f}}\left(\int_F (\bb v_h\times \bb n)\cdot \bb q_i dA\right)\bb \psi_{F,i}\\
&\hskip 4mm+\sum_{\substack{K\in \Tau_h\\ 1\leq i\leq N_v}}\left(\int_K \bb v_h\cdot \bb r_i dV\right)\bb \psi_{K,i},
\end{align*}
where $\{p_i\}_{i=1}^{N_e}\subset \mathbb{P}_{k-1}, \{\bb q_i\}_{i=1}^{N_f} \subset [\mathbb{P}_{k-2}]^3, \{\bb r_i\}_{i=1}^{N_v} \subset [\mathbb{P}_{k-3}]^3$ are some fixed polynomial basis, and the face and volume degrees of freedom are present only when $k\geq 2$ and $k\geq 3$, respectively. 
Denoting $\mathcal{F}^\partial(\Omega_h)$ and $\mathcal{E}^\partial(\Omega_h)$ as the set of faces and edges on the boundary $\partial\Omega_h$ and $\mathcal{F}^I(\Omega_h):=\mathcal{F}(\Omega_h)\setminus \mathcal{F}^\partial(\Omega_h)$ and $\mathcal{E}^I(\Omega_h):=\mathcal{E}(\Omega_h)\setminus \mathcal{E}^\partial(\Omega_h)$ as the set of faces and edges on the interior part of $\Omega_h$, we can write $\bb v_h := \bb v_h^0 + {\bb v}_h^\partial$, where $\bb v_h^0$ and ${\bb v}_h^\partial$ are interior and boundary parts of $\bb v_h$ defined as
\begin{align*}
{\bb v}_h^\partial &:= \sum_{\substack{E\in \mathcal{E}^\partial(\Omega_h)\\1\leq i \leq N_e}}\left(\int_E \bb v_h\cdot \bb\tau p_i ds\right)\bb \psi_{E,i}+\sum_{\substack{F\in \mathcal{F}^\partial (\Omega_h)\\ 1\leq i\leq N_f}}\left(\int_F (\bb v_h\times \bb n)\cdot \bb q_i dA\right)\bb \psi_{F,i},\\
\bb v_h^0 &:= \bb v_h - {\bb v}_h^\partial \\
&= \sum_{\substack{E\in \mathcal{E}^I(\Omega_h)\\1\leq i \leq N_e}}\left(\int_E \bb v_h\cdot \bb\tau p_i ds\right)\bb \psi_{E,i}+\sum_{\substack{F\in \mathcal{F}^I(\Omega_h)\\ 1\leq i\leq N_f}}\left(\int_F (\bb v_h\times \bb n)\cdot \bb q_i dA\right)\bb \psi_{F,i}\\
&\hskip 6mm+\sum_{\substack{K\in \Tau_h\\ 1\leq i\leq N_v}}\left(\int_K \bb v_h\cdot \bb r_i dV\right)\bb \psi_{K,i}.
\end{align*}
We note that by unisolvency of the degrees of freedom for N\'ed\'elec elements, $\gamma_t(\bb v_h)=\gamma_t(\bb v_h^\partial)$, and so $\bb v_h^0 \in \Xhz^{(k)}$. Moreover, $\text{supp}({\bb v}_h^\partial)=\Omega_h\setminus{\Omega}_h^0$ where $\displaystyle\Omega_h^0:= \bigcup_{\substack{K\in\mathcal{T}_h, \\ \overline{K}\cap \partial \Omega_h = \emptyset}}\overline{K}$.

We are now in a position to prove a key theorem of a posterior error estimation for the $p$-curl problem.

\begin{theorem} \label{thm:main_thm1}
Consider a $C^{1,1}$ simply connected bounded domain $\Omega$ and a source term $\bb f\in C(I;H(\text{div}^0;\Omega))$. Let $\{\Tau_h\}_{h>0}$ be shape-regular triangulations satisfying the interior mesh property provided by Theorem \ref{thm:triangulation}. If $\bb u$ and $\bb u_h$ are respective solutions to \eqref{WF} and \eqref{DWF}, then there exists $C>0$ depending only on shape regularity condition of Theorem \ref{thm:triangulation} such that for any $\bb v\in \Wzcurl$, 
\begin{align}
&\rpair{\pdt(\bb u-\tilde{\bb u}_h), \bb v} + \pair{\mathcal{P}(\bb u)-\mathcal{P}(\tilde{\bb u}_h),\bb v} \leq \left(\bb f,\bb v\right)_{\Omega\setminus \Omega_h} + \Res(\bb u_h,(\Pi_h\bb v)^\partial; \Omega_h\setminus \Omega_h^0)  \label{resEstEqn} \\          
&   \hspace{30mm} + C \big( (\eta_d+\eta_n+\eta_{n,\partial}) \norm{\bb v}_{L^2(\Omega)}+ (\eta_i+\eta_t+\eta_{t,\partial}) \norm{\curl\bb v}_{L^2(\Omega)} \big), \nonumber
\end{align} where $(\Pi_h\bb v)^\partial$ is the boundary part of $\Pi_h\bb v \in \Xh$, the Sch\"oberl quasi-interpolant of $\bb v$, 
\begin{align}
\Res(\bb u_h, \bb v_h^\partial; \Omega_h\setminus\Omega_h^0) &:= (\bb f-\partial_t \bb u_h, \bb v_h^\partial)_{\Omega_h\setminus\Omega_h^0}-\langle\mathcal{P}(\bb u_h), \bb v_h^\partial\rangle_{\Omega_h\setminus\Omega_h^0}, \label{resDefn}\\
 \hspace{10mm} \eta_i^2 &:= \sum_{K\in\Tau_h} h_K^2\norm{ \bb f-\pdt\bb u_h-\curl (\rho(\curl \bb u_h)\curl \bb u_h)}_{L^2(K)}^2 , \nonumber\\
\hspace{10mm} \eta_d^2 &:= \sum_{K\in\Tau_h} h_K^2\norm{\div \pdt \bb u_h}_{L^2(K)}^2 , \nonumber\\
\hspace{10mm} \eta_n^2 &:= \sum_{F\in\mathcal{F}^I(\Omega_h)} h_F\norm{\jump{\gamma_n(\pdt\bb u_h)}}_{L^2(F)}^2, \nonumber \\
\hspace{10mm} \eta_t^2 &:= \sum_{F\in\mathcal{F}^I(\Omega_h)} h_F\norm{\jump{\gamma_t(\rho(\curl \bb u_h)\curl \bb u_h)}}_{L^2(F)}^2 , \nonumber\\
\hspace{10mm} \eta_{n,\partial}^2 &:= \sum_{F\in\mathcal{F}^\partial(\Omega_h)} h_F\norm{\gamma_n(\bb f - \pdt\bb u_h)}_{L^2(F)}^2, \nonumber \\
\hspace{10mm} \eta_{t,\partial}^2 &:= \sum_{F\in\mathcal{F}^\partial(\Omega_h)} h_F\norm{\gamma_t(\rho(\curl \bb u_h)\curl \bb u_h)}_{L^2(F)}^2.\nonumber
\end{align} Here, $\jump{\gamma_t(\bb v)}:=\bb n_1\times \bb v_1+\bb n_2\times \bb v_2$ and $\jump{\gamma_n(\bb v)}:=\bb n_1\cdot \bb v_1+\bb n_2\cdot \bb v_2$ denote the tangential and normal jump of $\bb v_1:= \bb v|_{K_1}$ and $\bb v_2:= \bb v|_{K_2}$ across a common face $F = K_1\cap K_2$ with exterior normals $\bb n_1, \bb n_2$.
%
\end{theorem} 

\begin{proof}
Since $\bb u_h\in \Xhz^{(k)}$ and $\Omega_h\subset \Omega$, we can extend by zero using Lemma \ref{lem:exten} so that $\tilde{\bb u}_h\in W^p_0(\text{curl};\Omega)$. It follows then for any $\bb v\in \Wzcurl$ and $\bb v_h=\bb v_h^0+\bb v_h^\partial \in \Xh^{(k)}$,
\begin{align}
& (\partial_t(\bb u-\tilde{\bb u}_h),\bb v)_{\Omega}+\left<\mathcal{P}(\bb u)-\mathcal{P}(\tilde{\bb u}_h),\bb v\right>_\Omega \nonumber \\
&= \underbrace{(\bb f,\bb v)_{\Omega}}_{\text{by Lemma \ref{ncTest}} }- \underbrace{\left[(\partial_t \tilde{\bb u}_h,\bb v)_{\Omega}+\left<\mathcal{P}(\tilde{\bb u}_h),\bb v\right>_\Omega\right]}_{= (\partial_t \bb u_h,\bb v)_{\Omega_h}+\left<\mathcal{P}(\bb u_h),\bb v\right>_{\Omega_h}} - \underbrace{\left[(\bb f-\partial_t \bb u_h,\bb v_h^0)_{\Omega_h}-\left<\mathcal{P}(\bb u_h),\bb v_h^0\right>_{\Omega_h}\right]}_{=0 \text{ by \eqref{DWF} since } \bb v_h^0\in \Xhz^{(k)}}\nonumber\\
&= (\bb f,\bb v)_{\Omega\setminus\Omega_h} + (\bb f-\partial_t \bb u_h,\bb v)_{\Omega_h}-\left<\mathcal{P}(\bb u_h),\bb v\right>_{\Omega_h} \nonumber\\
&\hskip 4mm - \left[(\bb f-\partial_t \bb u_h,\bb v_{h}-\bb v_h^\partial)_{\Omega_h}-\left<\mathcal{P}(\bb u_h),\bb v_{h}-\bb v_h^\partial\right>_{\Omega_h}\right]\nonumber\\
&=(\bb f,\bb v)_{\Omega\setminus\Omega_h}+(\bb f-\partial_t \bb u_h,\bb v-\bb v_{h})_{\Omega_h}-\left<\mathcal{P}(\bb u_h),\bb v- \bb v_{h}\right>_{\Omega_h}+\Res(\bb u_h, \bb v_h^\partial; \Omega_h\setminus\Omega_h^0) \label{ee0}
\end{align}
Since $p \geq 2$, the restriction $\bb v \in \Wcurlh \subset H(\text{curl};\Omega_h)$ and so we can set $\bb v_h$ to be the quasi-interpolant $\bb v_h := \Pi_h \bb v $ of Theorem \ref{quasiHyp}. Moreover, there exists $\phi \in H^1(\Omega_h)$ and $\bb w\in H^1(\Omega_h)^3$
for which $\bb v-\Pi_h \bb v=\grad\phi + \bb w$ and the estimates \eqref{quasiE1} and \eqref{quasiE2} hold. 
It remains to estimate $(\bb f-\partial_t \bb u_h,\bb v-\bb v_{h})_{\Omega_h}-\left<\mathcal{P}(\bb u_h),\bb v- \bb v_{h}\right>_{\Omega_h}$. For this, we apply Green's formula \eqref{tangTrace} and \eqref{normalTrace} to obtain
\begin{align}
&(\bb f-\partial_t \bb u_h,\bb v-\bb v_{h})_{\Omega_h}-\left<\mathcal{P}(\bb u_h),\bb v- \bb v_{h}\right>_{\Omega_h}\nonumber\\
&=\sum_{K\in\Tau_h} \rpairk{\bb f -\pdt\bb u_h,\grad \phi + \bb w} - \rpairk{\rho(\curl\bb u_h)\curl\bb u_h,\curl(\grad \phi + \bb w)} \nonumber\\
&=\sum_{K\in\Tau_h}\left[\frac{}{} \rpairk{\bb f -\pdt\bb u_h,\bb w}- \rpairk{\div(\bb f -\pdt\bb u_h),\phi} +\left(\gamma_n(\bb f -\pdt\bb u_h),\gamma_0(\phi)\right)_{\partial K} \right.\nonumber \\
&\hspace{15mm} \left.\frac{}{}- \rpairk{\curl(\rho(\curl\bb u_h)\curl\bb u_h),\bb w}-\left( \gamma_t(\rho(\curl\bb u_h)\curl\bb u_h),\gamma_0(\bb w)\right)_{\partial K}\right] \nonumber
\end{align} 
\begin{align}
&=\sum_{K\in\Tau_h} \left[\frac{}{}\rpairk{\bb f -\pdt\bb u_h-\curl (\rho(\curl\bb u_h)\curl\bb u_h),\bb w}-\rpairk{\div(\bb f -\pdt\bb u_h),\phi}\right]\nonumber \\
&\hspace{4mm}+\sum_{F\in \mathcal{F}^I(\Omega_h)} \left[\frac{}{}\right.\underbrace{\frac{}{}\rpaire{\frac{}{}\jump{\gamma_n(\bb f-\pdt\bb u_h)},\gamma_0(\phi)}}_{\substack{=\rpaire{\frac{}{}\jump{\gamma_n(-\pdt\bb u_h)},\gamma_0(\phi)},\\ \text{ since }\bb f(t)\in H(\text{div};\Omega)}} +  \rpaire{\frac{}{}\jump{\gamma_t(\rho(\curl\bb u_h)\curl\bb u_h)},\gamma_0(\bb w)}\left.\frac{}{}\right]\nonumber\\&\hspace{4mm}+\sum_{F\in\mathcal{F}^\partial(\Omega_h)}\left[\rpaire{\frac{}{}\gamma_n(\bb f-\partial_t \bb u_h),\gamma_0(\phi)}+\rpaire{\frac{}{}\gamma_t(\rho(\curl\bb u_h)\curl\bb u_h),\gamma_0(\bb w)}\right]\nonumber \\
&= \sum_{K\in\Tau_h} R_{K,i}(\bb u_h;\bb w) + R_{K,d}(\bb u_h;\phi) + \sum_{F\in\mathcal{F}^I(\Omega_h)} R_{F,n}(\bb u_h;\phi) + R_{F,t}(\bb u_h;\bb w)\hspace{10mm}\nonumber\\
&\hspace{54mm}+\sum_{F\in\mathcal{F}^\partial(\Omega_h)} R_{F,n}^\partial(\bb u_h;\phi) + R_{F,t}^\partial(\bb u_h;\bb w),  \label{residualE2}
\end{align} 
where the residuals are defined by 
\begin{align*}
R_{K,i}(\bb u_h;\bb w) & := \rpairk{\bb f -\pdt\bb u_h-\curl \left(\rho(\curl\bb u_h)\curl\bb u_h\right) ,  \bb w}, \\
R_{K,d}(\bb u_h;\phi) & := -\rpairk{\div(\bb f -\pdt\bb u_h),\phi},\\
R_{F,n}(\bb u_h;\phi) & := \rpaire{\jump{\gamma_n(-\pdt\bb u_h)},\gamma_0(\phi)},\\
R_{F,t}(\bb u_h;\bb w) & := \rpaire{\jump{\gamma_t(\rho(\curl\bb u_h)\curl\bb u_h)},\gamma_0(\bb w)},\\
R_{F,n}^\partial(\bb u_h;\phi) & := \rpaire{\gamma_n(\bb f-\pdt\bb u_h),\gamma_0(\phi)},\\
R_{F,t}^\partial(\bb u_h;\bb w) & := \rpaire{\gamma_t(\rho(\curl\bb u_h)\curl\bb u_h),\gamma_0(\bb w)}.
\end{align*}
Indeed, $R_{K,i}$ is the standard interior local residual term while $R_{F,n}$ and $R_{F,t}$ measure respectively the normal and tangential discontinuity of $\gamma_n(-\pdt\bb u_h)$ and $\gamma_t(\rho(\curl\bb u_h)\curl\bb u_h))$ across neighbouring elements. Moreover, $R_{F,n}^\partial$ and $R_{F,t}^\partial$ measures the boundary defects of $\gamma_n(\bb f-\pdt\bb u_h)$ and $\gamma_t(\rho(\curl\bb u_h)\curl\bb u_h)$ along the boundary faces of $\partial\Omega_h$. We observe that at each $t$, $\bb f \in H(\text{div}^0;\Omega)$ implies that the first
term in $R_{K,d}$ satisfies $ \rpairk{\div \bb f,\phi} = 0$ but the second term $\div \bb u_h$ vanishes only for first order N\'ed\'elec elements. Hence, the residual $R_{K,d}$ measures the defect in the divergence constraint at the discrete level, namely by
\begin{align*}
\sum_{K\in \Tau_h} R_{K,d}(\bb u_h;\phi) = \sum_{K\in \Tau_h} \rpairk{\div\pdt\bb  u_h,\phi}.
\end{align*}

Next, we proceed to estimate each term in the sum of \eqref{residualE2} by using Holder's inequality, \eqref{quasiE1}, and \eqref{quasiE2}. We use the convention that the constant $C$ may change from one line to the next and only depends on the shape-regularity of $\{\Tau_h\}_{h>0}$.
\begin{align}
\sum_{K\in\Tau_h}R_{K,i}(\bb u_h;\bb w) 
& \leq \sum_{K\in\Tau_h}\norm{ \bb f -\pdt \bb u_h-\curl \left(\rho(\curl\bb u_h)\curl\bb u_h\right)}_{L^2(K)}\norm{\bb w}_{L^2(K)}\nonumber\\ 
& \hspace{-18mm}\leq C\sum_{K\in\Tau_h}h_K \norm{ \bb f -\pdt \bb u_h-\curl\left(\rho(\curl\bb u_h)\curl\bb u_h\right)}_{L^2(K)}\norm{\curl \bb v}_{L^2(\omega_K)} \nonumber \\ 
&\hspace{-18mm}\leq C\left(\sum_{K\in\Tau_h} h_K^2\norm{ \bb f -\pdt \bb u_h-\curl \left(\rho(\curl\bb u_h)\curl\bb u_h\right)}_{L^2(K)}^2\right)^{1/2}\norm{\curl \bb v}_{L^2(\Omega)} \label{ee1}
\end{align}
To bound the $R_{K,d}$ terms, we proceed in the same way
\begin{align}
\sum_{K \in \Tau_h} R_{K,d}( \bb u_h; \bb w) & \leq \sum_{K \in \Tau_h} \norm{ \nabla \cdot \partial_t \bb u_h }_{L^2(K)} \norm{\phi}_{L^2(K)} \nonumber \\
& \leq C \sum_{K \in \Tau_h} h_K \norm{\nabla \cdot \partial \bb u_h}_{L^2(K)} \norm{ \bb v }_{L^2(\omega_K)} \nonumber \\
& \leq C \Big( \sum_{K \in \Tau_h} h_K^2 \norm{\nabla \cdot \partial \bb u_h}_{L^2(K)}^2 \Big)^{1/2} \norm{\bb v}_{L^2(\Omega)} . \label{ee2}
\end{align}
For the $R_{F,t}(\bb u_h;\bb w)$ terms, we begin with a stability estimate. Using \eqref{scaledtraceThm} and $h_F\simeq h_K$ 
for shape-regular $\{\Tau_h\}_{h>0}$, 
we find
\begin{align*}
\norm{\gamma_0(\bb w)}_{L^2(F)} &\leq C\left(h_F^{-1}\norm{\bb w}_{L^2(K)}^2+h_F \norm{\nabla\bb w}_{L^2(K)}^2\right)^{1/2} \\
& \leq C\left(h_F^{-1}h_K^2\norm{\curl\bb v}_{L^2(\omega_K)}^2+h_F \norm{\curl\bb v}_{L^2(\omega_K)}^2\right)^{1/2}\\
& \leq Ch_F^{1/2}\norm{\curl\bb v}_{L^2(\omega_K)}.
\end{align*}
Employing this last estimate, we obtain
\begin{align}
\sum_{F\in\mathcal{F}^I(\Omega_h)} R_{F,t}(\bb u_h;\bb w) & \leq \sum_{F\in\mathcal{F}^I(\Omega_h)} \norm{\jump{\gamma_t(\rho(\curl\bb u_h)\curl\bb u_h)}}_{L^2(F)}\norm{\gamma_0(\bb w)}_{L^2(F)} \nonumber \\
& \hspace{-16mm}\leq C\sum_{F\in\mathcal{F}^I(\Omega_h)} h_F^{1/2}\norm{\jump{\gamma_t(\rho(\curl\bb u_h)\curl\bb u_h)}}_{L^2(F)}\norm{\curl \bb v}_{L^2(\omega_K)} \nonumber \\
& \hspace{-16mm} \leq C\left(\sum_{F\in\mathcal{F}^I(\Omega_h)} h_F\norm{\jump{\gamma_t(\rho(\curl\bb u_h)\curl\bb u_h)}}_{L^2(F)}^2\right)^{1/2}\norm{\curl\bb v}_{L^2(\Omega)}. \label{ee3}
\end{align}
We can bound the boundary terms $R_{F,t}^\partial(\bb u_h;\bb w)$ in the same manner and obtain
\begin{align}
\sum_{F\in\mathcal{F}^\partial(\Omega_h)} R_{F,t}^\partial(\bb u_h;\bb w) &\nonumber\\
&\hspace{-16mm}\leq C\left(\sum_{F\in\mathcal{F}^\partial(\Omega_h)} h_F\norm{\gamma_t(\rho(\curl\bb u_h)\curl\bb u_h)}_{L^2(F)}^2\right)^{1/2}\norm{\curl\bb v}_{L^2(\Omega)}. \label{ee4}
\end{align}
Similarly to the previous stability estimate, using \eqref{scaledtraceThm} and the shape-regularity of $\{\Tau_h\}_{h>0}$, one can show that
\begin{align*}
\norm{\gamma_0(\phi)}_{L^2(F)} &\leq Ch_F^{1/2}\norm{\bb v}_{L^2(\omega_K)}.
\end{align*}
Applying this to the $R_{F,n}(\bb u_h;\phi)$ term, one finds
\begin{align}
\sum_{F\in\mathcal{F}^I(\Omega_h)} R_{F,n}(\bb u_h;\phi) & \leq \sum_{F\in\mathcal{F}^I(\Omega_h)} \norm{\jump{\gamma_n(\pdt\bb u_h)}}_{L^2(F)}\norm{\gamma_0(\phi)}_{L^2(F)} \nonumber \\
& \leq C\sum_{F\in\mathcal{F}^I(\Omega_h)} h_F^{1/2}\norm{\jump{\gamma_n(\pdt\bb u_h)}}_{L^2(F)}\norm{\bb v}_{L^2(\omega_K)} \nonumber \\
& \leq C\left(\sum_{F\in\mathcal{F}^I(\Omega_h)} h_F\norm{\jump{\gamma_n(\pdt\bb u_h)}}_{L^2(F)}^2\right)^{1/2}\norm{\bb v}_{L^2(\Omega)}. \label{ee5}
\end{align}
Similarly, we can bound the boundary terms $R_{F,n}^\partial(\bb u_h;\phi)$ and obtain
\begin{align}
\sum_{F\in\mathcal{F}^\partial(\Omega_h)} R_{F,n}^\partial(\bb u_h;\phi) \leq C\left(\sum_{F\in\mathcal{F}^\partial(\Omega_h)} h_F\norm{\gamma_n(\bb f- \pdt\bb u_h)}_{L^2(F)}^2\right)^{1/2}\norm{\bb v}_{L^2(\Omega)}. \label{ee6}
\end{align}
Thus, combining \eqref{ee0}-\eqref{ee6}, we have shown the desired result.
\end{proof}

Now we show the a posteriori error estimators in Theorem \ref{thm:main_thm1} are reliable in the following sense. 

\begin{theorem} \label{thm:main_thm2}
Let $\bb u$, $\bb u_h$ and $\bb f$ be as stated in Theorem \ref{thm:main_thm1} and denote the error as $\bb e:=\bb u-\tilde{\bb u}_h$ and $\bb e_0=\bb e|_{t=0}$. Then there exists some positive constants $C_1(p,\alpha)$ and $C_2(p,T)$ such that,
\begin{align*}
&\sup_{s\in[0,T]}\norm{\bb e(s)}^2_{L^2(\Omega)}+C_1\int_0^T\norm{\nabla \times\bb e(s)}^p_{L^p(\Omega)}ds \\
&\hspace{4mm}\leq C_2\left(\frac{}{}\int_0^T \NC_1(\bb f(s); \Omega\setminus \Omega_h)^2+\NC_2(\bb f(s),\bb u_h(s); \Omega_h\setminus \Omega_h^0)^{2} ds\right. \nonumber\\
&\hspace{16mm} \left.+\norm{\bb e_0}^2_{L^2(\Omega)}+\int_0^T \big(\eta_d^2(s)+\eta_n^2(s)+\eta_{n,\partial}^2(s)+\eta_i^q(s)+\eta_t^q(s)+\eta_{t,\partial}^q(s)\big)ds\right) ,
\end{align*} where $\NC_1$ and $\NC_2$ are non-conforming geometric errors defined as,
\begin{align*}
&\NC_1^2\equiv\NC_1(\bb f(t); \Omega\setminus \Omega_h)^2 := \norm{\bb f(t)}_{L^2(\Omega\setminus \Omega_h)}^2, \\
&\NC_2^2\equiv\NC_2(\bb f(t), \bb u_h(t); \Omega_h\setminus \Omega_h^0)^2 := \norm{\bb f(t)-\partial_t \bb u_h(t)}_{L^2(\Omega_h\setminus \Omega_h^0)}^2+h_\partial^\frac{2p}{p-2}\\
&\hskip 60mm + \alpha \norm{\nabla \times \bb u_h(t)}^{p}_{L^{2(p-1)}(\Omega_h\setminus \Omega_h^0)},
\end{align*} 
with $\displaystyle h_\partial = \max_{\substack{K\in \mathcal{T}_h\\ K \subset \Omega_h\setminus \Omega_h^0}} h_K$.
\end{theorem} 

Above, $\NC_1$ and $\NC_2$ are called non-conforming geometric errors since $\NC_1 = \mathcal{O}(\operatorname{vol}(\Omega\setminus \Omega_h))$ and $\NC_1 = \mathcal{O}(\operatorname{vol}(\Omega_h\setminus \Omega_h^0))$. Specifically, $\NC_1$ measures with the geometric defect of $\bb f$ between the embedded polyhedral $\Omega_h$ domain and the $C^{1,1}$ domain, while $\NC_2$ arises from the boundary data defect of $\bb e$ along $\partial \Omega_h$.


\begin{proof}
Since $\tilde{\bb u}_h \in \Wzcurl$, setting $\bb v=\bb e \in \Wzcurl$ in \eqref{resEstEqn} gives,
\begin{align}
&\frac{d}{dt}\frac{1}{2}\norm{\bb e}_{L^2(\Omega)}^2 + \pair{\mathcal{P}(\bb u)-\mathcal{P}(\tilde{\bb u}_h),\bb e}  \label{eq:resEstEqn} \\          
& \hspace{16mm} \leq \left(\bb f,\bb e\right)_{\Omega\setminus \Omega_h} + \Res(\bb u_h,(\Pi_h\bb e)^\partial; \Omega_h\setminus \Omega_h^0) \nonumber \\
&   \hspace{20mm} + C \big( (\eta_d+\eta_n+\eta_{n,\partial}) \norm{\bb e}_{L^2(\Omega)}+ (\eta_i+\eta_t+\eta_{t,\partial}) \norm{\curl\bb e}_{L^2(\Omega)} \big) \nonumber
\end{align}
We proceed to estimate each term on both sides of \eqref{eq:resEstEqn}.
First, we can bound from below the second term on the left hand side of \eqref{eq:resEstEqn} by the following inequality \cite[eqn 24]{Cho89}, where for some $C_p>0$,
\begin{equation*}
C_p|\bb x-\bb y|^p\leq(|\bb x|^{p-2}\bb x-|\bb y|^{p-2}\bb y)\cdot(\bb x-\bb y), \quad \forall \bb x,\bb y \in \mathbb{R}^d.
\end{equation*}
Thus, setting $\bb x=\curl \bb u,$ $\bb y=\curl \tilde{\bb u}_h$ and integrating the above inequality above gives the coercivity estimate
\begin{equation}
C_p \alpha \norm{\curl \bb e}_{L^p(\Omega)}^p\leq \pair{\mathcal{P}(\bb u)-\mathcal{P}(\tilde{\bb u}_h),\bb e} . \label{coercivity}
\end{equation} 
Second, to bound from above the residual term on the right hand side of \eqref{eq:resEstEqn}, note that $(\cdot)^\partial: \Xh^{(k)}\rightarrow \Xh^{(k)}$ is a projection on a finite dimensional space and thus is bounded on $H(\text{curl};\Omega_h)$ with the operator norm $C_\partial$. Moreover, by \eqref{quasiE1} and \eqref{quasiE2}, there are constants $c_1, c_2>0$ such that for each $K\in \Tau_h$,
\begin{align}
\norm{\Pi_h \bb e}_{L^2(K)} &\leq \norm{\bb e}_{L^2(K)} +\norm{\Pi_h \bb e-\bb e}_{L^2(K)} \label{eq:SchEst1}\\
&\leq \norm{\bb e}_{L^2(K)}+\norm{\nabla \phi}_{L^2(K)}+\norm{\bb w}_{L^2(K)} \nonumber\\
&\leq c_1(\norm{\bb e}_{L^2(\omega_K)}+h_K\norm{\nabla \times \bb e}_{L^2(\omega_K)})  \nonumber\\
\norm{\nabla \times\Pi_h \bb e}_{L^2(K)} &\leq \norm{\nabla \times\bb e}_{L^2(K)} +\norm{\nabla \times(\Pi_h \bb e-\bb e)}_{L^2(K)} \label{eq:SchEst2}\\
&\leq \norm{\nabla \times \bb e}_{L^2(K)}+\norm{\nabla \times \nabla \phi}_{L^2(K)}+\norm{\nabla \times \bb w}_{L^2(K)} \nonumber \\
&\leq  c_2 \norm{\nabla \times \bb e}_{L^2(\omega_K)} \nonumber
\end{align} And since each $K$ overlaps with finitely many $\omega_K$, $\Pi_h: H(\text{curl};\Omega_h) \rightarrow \Xh^{(k)}$ is bounded as \eqref{eq:SchEst1} and \eqref{eq:SchEst2} implies for some positive constant $C$,
\begin{align}
\norm{\Pi_h \bb e}_{L^2(\Omega_h\setminus \Omega_h^0)}^2 &= \sum_{\substack{K\in \Tau_h\\ K\subset \Omega_h\setminus \Omega_h^0}} \norm{\Pi_h \bb e}_{L^2(K)}^2 \label{eq:SchBdryEst1} \\
&\leq 2c_1^2 \sum_{\substack{K\in \Tau_h\\ K\subset \Omega_h\setminus \Omega_h^0}} \left(\norm{\bb e}_{L^2(\omega_K)}^2+h_K^2\norm{\nabla \times \bb e}_{L^2(\omega_K)}^2 \right) \nonumber\\
& \leq C\left(\norm{\bb e}_{L^2(\Omega_h\setminus \Omega_h^0)}^2+h_\partial^2\norm{\nabla\times \bb e}_{L^2(\Omega_h\setminus \Omega_h^0)}^2\right) \nonumber \\
\norm{\nabla\times\Pi_h \bb e}_{L^2(\Omega_h\setminus \Omega_h^0)}^2 & 
          \leq c_2^2\sum_{\substack{K\in \Tau_h\\ K\subset \Omega_h\setminus \Omega_h^0}} \norm{\nabla \times \bb e}_{L^2(\omega_K)}^2 
          \leq C\norm{\nabla \times \bb e}_{L^2(\Omega_h\setminus \Omega_h^0)}^2 \label{eq:SchBdryEst2}
\end{align}
Using \eqref{eq:SchBdryEst1}, the first residual term in \eqref{resDefn} can be bounded above with the help of Young's inequality for any $\epsilon>0$ and some positive constant $C$,
\begin{align*}
&\left|(\bb f-\partial_t \bb u_h, (\Pi_h\bb e)^\partial)_{\Omega_h\setminus\Omega_h^0}\right|\leq C_\partial\norm{\bb f-\partial_t \bb u_h}_{L^2(\Omega_h\setminus \Omega_h^0)}\norm{\Pi_h\bb e}_{L^2(\Omega_h\setminus \Omega_h^0)}\\
& \hskip 2mm\leq C\left( \norm{\bb f-\partial_t \bb u_h}_{L^2(\Omega_h\setminus \Omega_h^0)}^2+\norm{\bb e}_{L^2(\Omega_h\setminus \Omega_h^0)}^2+h_\partial^2\norm{\nabla\times \bb e}_{L^2(\Omega_h\setminus \Omega_h^0)}^2\right)\\
& \hskip 2mm\leq C\left( \norm{\bb f-\partial_t \bb u_h}_{L^2(\Omega_h\setminus \Omega_h^0)}^2+\norm{\bb e}_{L^2(\Omega)}^2+\frac{1}{q'\epsilon^{q'}}h_\partial^{2q'}+\frac{2\epsilon^{\frac{p}{2}}C_{q',h}^p}{p}\norm{\nabla\times \bb e}_{L^p(\Omega)}^p\right)
\end{align*}
where $q' = \frac{p}{p-2}$ and the last inequality follows from 
$\| \bb e \|_{L^2(\Omega_h\setminus \Omega_h^0)} \leq \| \bb e \|_{L^2(\Omega_h)} \leq \| \bb e \|_{L^2(\Omega)}$ and 
$\| \curl \bb e \|_{L^2(\Omega_h\setminus \Omega_h^0)} \leq C_{q',h} \| \curl \bb e \|_{L^p(\Omega_h\setminus \Omega_h^0)} 
    \leq C_{q',h} \| \curl \bb e \|_{L^p(\Omega)}$ with $C_{q',h} = \text{vol}(\Omega_h\setminus \Omega_h^0)^{\frac{1}{2q'}}$.
Similarly, using \eqref{eq:SchBdryEst2}, for an arbitrary positive $\epsilon$, the second residual term in \eqref{resDefn} can be bounded above as
\begin{align*}
& \left|\langle\mathcal{P}(\bb u_h), (\Pi_h \bb e)^\partial\rangle_{\Omega_h\setminus\Omega_h^0}\right| \leq \alpha \int_{\Omega_h\setminus \Omega_h^0} |\nabla\times \bb u_h|^{p-1}|(\nabla \times \Pi_h \bb e)^{\partial}|dV \\
& \hskip 2mm\leq \alpha \| |\nabla \times \bb u_h|^{p-1} \|_{L^{2}(\Omega_h\setminus\Omega_h^0)}\|(\nabla \times \bb e)^\partial\|_{L^{2}(\Omega_h\setminus\Omega_h^0)}\\
& \hskip 2mm\leq \alpha C_\partial \| \nabla \times \bb u_h\|_{L^{2(p-1)}(\Omega_h\setminus\Omega_h^0)}^{p-1}\|\nabla \times \bb e\|_{L^{2}(\Omega_h\setminus\Omega_h^0)} \\
& \hskip 2mm\leq \alpha C_\partial \left( \frac{1}{q\epsilon^q} \| \nabla \times \bb u_h\|_{L^{2(p-1)}(\Omega_h\setminus\Omega_h^0)}^{p}+ \frac{\epsilon^{p}C_{q,h}^p}{p}\norm{\nabla\times \bb e}_{L^p(\Omega)}^p\right),
\end{align*}where $(p-1)q = p$ and the last step follows from $\| \curl \bb e \|_{L^2(\Omega_h\setminus \Omega_h^0)} \leq C_{q,h} \| \curl \bb e \|_{L^p(\Omega)}$ with $C_{q,h} = \text{vol}(\Omega_h\setminus \Omega_h^0)^{\frac{1}{2q}}$.
Combining these two estimates for the residual of \eqref{resDefn}, we have for $0<\epsilon<1$ that for some positive constant $C'$ depending on $\text{vol}(\Omega_h\setminus \Omega_h^0), p, C, C_\partial, \epsilon$ and some positive constant $C''$ depending on $\text{vol}(\Omega_h\setminus \Omega_h^0), \alpha, p, C_\partial$ but not $\epsilon$,
\begin{align*}
\left|\Res(\bb u_h,(\Pi_h\bb e)^\partial; \Omega_h\setminus \Omega_h^0)\right| \leq C'\NC_2^2+C''\left(\norm{\bb e}_{L^2(\Omega)}^2+\epsilon^{\frac{p}{2}}\norm{\nabla\times \bb e}_{L^p(\Omega)}^p\right)
\end{align*}
Finally, combining with \eqref{coercivity}, \eqref{eq:resEstEqn} becomes,
\begin{align} \label{ineq1}
 &\frac{d}{dt}\frac{1}{2}\norm{\bb e}_{L^2(\Omega)}^2+C_p\alpha \norm{\curl \bb e}_{L^p(\Omega)}^p \\
 &\hspace{4mm}\leq \NC_1^2+C'\NC_2^2+C''\left(\norm{\bb e}_{L^2(\Omega)}^2+\epsilon^{\frac{p}{2}}\norm{\nabla\times \bb e}_{L^p(\Omega)}^p\right)\nonumber\\
 &\hspace{8mm} + C\left(\frac{1}{2}\left(\eta_d^2+\eta_n^2+\eta_{n,\partial}^2\right) + \frac{3}{2}\norm{\bb e}^2_{L^2(\Omega)}\right.\nonumber \\
 &\hspace{18mm} \left.+\frac{1}{q\epsilon^q}(\eta_i^q+\eta_t^q+\eta_{t,\partial}^q)+\frac{3\epsilon^p}{p}\norm{\curl \bb e}^p_{L^p(\Omega)}\right)\nonumber
\end{align} 
Thus, for sufficiently small $\epsilon$, inequality \eqref{ineq1} implies that there exists positive constants $C_1(p,\alpha,\epsilon)$ and $a(C,C',C'',p,\epsilon)$ for which
\begin{align*}
\frac{d}{dt}\norm{\bb e}_{L^2(\Omega)}^2+C_1\norm{\curl \bb e}_{L^p(\Omega)}^p &\leq a\left(\NC_1^2 +\NC_2^2\right.\\
&\hskip 10mm\left.\norm{\bb e}^2_{L^2(\Omega)}+\eta_d^2+\eta_n^2+\eta_{n,\partial}^2+\eta_i^q+\eta_t^q+\eta_{t,\partial}^q\right).
\end{align*} So multiplying by $e^{-at}$ and integrating yields
\begin{align}
&\norm{\bb e(t)}_{L^2(\Omega)}^2+C_1\int_0^t  e^{a(t-s)}\norm{\curl \bb e(s)}_{L^p(\Omega)}^p ds \nonumber\\ 
 &\hspace{5mm} \leq e^{a t}\norm{\bb e_0}^2_{L^2(\Omega)}+a\int_0^t e^{a(t-s)}\left(\NC_1^2(s)+\NC_2^{2}(s)\right.\nonumber\\
  &\hspace{45mm} +\left.\eta_d^2(s)+\eta_n^2(s)+\eta_{n,\partial}^2(s)+\eta_i^q(s)+\eta_t^q(s)+\eta_{t,\partial}^q(s)\right)ds. \nonumber \\
&\Rightarrow \norm{\bb e(t)}_{L^2(\Omega)}^2+C_1\int_0^t  \norm{\curl \bb e(s)}_{L^p(\Omega)}^p ds  \nonumber\\
&\hspace{5mm} \leq C_2\left(\frac{}{}\int_0^T\NC_1^2(s) +\NC_2^{2}(s)ds +\norm{\bb e_0}^2_{L^2(\Omega)} \right. \nonumber\\
&\hspace{16mm} \left.+\int_0^T \big(\eta_d^2(s)+\eta_n^2(s)+\eta_{n,\partial}^2(s)+\eta_i^q(s)+\eta_t^q(s)+\eta_{t,\partial}^q(s)\big)ds\right) \label{ineq2}
\end{align} since $1\leq e^{a(t-s)} \leq e^{aT}$ for $0\leq s \leq t \leq T$ with $C_2(p,T)=\max\{1,a\}e^{a T}$. Taking the supremum over all $t\in [0,T]$ of equation \eqref{ineq2} gives the desired result.
\end{proof}

\section{A posteriori error estimate for AC loss}
\label{sec:ACloss}

For many engineering applications, the quantity of interest is the AC loss over one period $T$,
$$ 
        Q(\bb u):=\frac{1}{T}\int_0^T\norm{\curl \bb u(s)}_{L^p(\Omega)}^p ds. 
$$
In particular, we wish to derive a posteriori error estimates for $|Q(\bb u)-Q(\boldsymbol u_h)|$. To do this, we first derive the following elementary estimate and subsequently use it to show the error for $Q$ is related to the a posteriori error estimates derived previously.
\begin{lemma}Assume $1\leq p$ and $M>0$, then for any functions $x:[0,T]\rightarrow [0,M]$, $y:[0,T]\rightarrow [0,M]$, we have 
\begin{align*}
\int_0^T|x(t)^p-y(t)^p| dt\leq p T^{1-\frac{1}{p}}M^{p-1}\left(\int_0^T|x(t)-y(t)|^p\right)^{1/p} .
\end{align*} \label{p4}
\end{lemma}
\begin{proof}
For any $t\in [0,T]$, applying the mean value theorem for the function $f(z)=z^p$ on $[0,M]$ implies there exists $\xi(t)\in (0,M)$ satisfying
\begin{align*}
|x(t)^p-y(t)^p|=|x(t)-y(t)|\cdot p\xi(t)^{p-1}\leq pM^{p-1}|x(t)-y(t)| .
\end{align*}Thus, integrating over $[0,T]$ gives,
\begin{align*}
\int_0^T |x(t)^p-y(t)^p|dt&\leq pM^{p-1}\int_0^T|x(t)-y(t)| dt \\
&\leq p T^{1-\frac{1}{p}}M^{p-1}\left(\int_0^T|x(t)-y(t)|^p\right)^{1/p} .
\end{align*}
\end{proof}


\begin{theorem} Let $\bb u$, $\bb u_h$, the error $\bb e, \bb e_0$ and positive constants $C_1, C_2$ be as stated in Theorem \ref{thm:main_thm1}. Let $M$ be the maximum of the stability bounds for the weak formulation \eqref{WF} and \eqref{DWF} given by equations \eqref{ineq:stability-continuous} and \eqref{stabEst2}. Then, the following inequality holds.
\begin{align*}
&|Q(\bb u)-Q(\tilde{\bb u}_h)|\nonumber\\
&\hspace{2mm}\leq pT^{-\frac{1}{p}}M^{p-1}\left(\int_0^T\norm{\curl \bb e(s)}_{L^p{(\Omega})}^p ds\right)^{1/p} \\
&\hspace{2mm}\leq C\left(\int_0^T \NC_1(\bb f(s); \Omega\setminus \Omega_h)^2+\NC_2(\bb f(s),\bb u_h(s); \Omega_h\setminus \Omega_h^0)^{2} ds\right.\\
&\left.\hspace{11mm}+\norm{\bb e_0}^2_{L^2(\Omega)}+\int_0^T \big(\eta_d^2(s)+\eta_n^2(s)+\eta_{n,\partial}^2(s)+\eta_i^q(s)+\eta_t^q(s)+\eta_{t,\partial}^q(s)\big)ds\right)^{1/p},
\end{align*} where $C:=\frac{C_2}{C_1}pT^{-\frac{1}{p}}M^{p-1}$.
\end{theorem}
\begin{proof} Let $x(t):=\norm{\curl \bb u}_{L^p(\Omega)}$ and $y(t):=\norm{\curl \tilde{\bb u}_h}_{L^p(\Omega)}$,
which we know are bounded by inequalities \eqref{ineq:stability-continuous} and \eqref{stabEst2}. 

Since $0\leq \norm{\curl\bb u}_{L^p(\Omega)}\leq M$, $0\leq\norm{\curl\tilde{\bb u}_h}_{L^p(\Omega)}\leq M$ for $0\leq t\leq T$, Lemma \ref{p4} implies,
 \begin{align}
|Q(\bb u)-Q(\boldsymbol u_h)|&=\frac{1}{T}\left|\int_0^T (x(t)^p-y(t)^p)dt\right| \leq \frac{1}{T}\int_0^T \left|x(t)^p-y(t)^p\right|dt \nonumber \\
&\leq pT^{-\frac{1}{p}}M^{p-1}\left(\int_0^T|x(t)-y(t)|^p dt\right)^{1/p} \label{e10}
\end{align}
Since $|x(t)-y(t)|=\left|\norm{\curl \bb u}_{L^p(\Omega)}-\norm{\curl \tilde{\bb u}_h}_{L^p(\Omega)}\right|\leq \norm{\curl \bb e}_{L^p(\Omega)}$, then by monotonicity of $f(z)=z^p$, we have $|x(t)-y(t)|^p\leq \norm{\curl \bb e}_{L^p(\Omega)}^p$.
Thus, again by monotonicity of $f(z)=z^{1/p}$,
\begin{align} 
\left(\int_0^T |x(t)-y(t)|^p dt\right)^{1/p}\leq \left(\int_0^T\norm{\curl \bb e(s)}_{L^p(\Omega)}^p ds\right)^{1/p} \label{e11}
\end{align} Combining inequalities \eqref{e10}, \eqref{e11} and Theorem \ref{thm:main_thm2} yield the desired result.
\end{proof}

\section{Numerical results}

We present numerical results in 2D supporting the reliability of the error estimators presented in Section \ref{sec:apee}. In the following, the $p$-curl problem is discretized in space using first order N\'ed\'elec elements and in time using the backward Euler method. While higher order time stepping schemes can be used, the discretization error is shown to be dominated by the spatial errors due to the low order approximation of first order N\'ed\'elec elements. 
The fully discrete formulation was implemented in Python using the FEniCS package \cite{Fenics15}. For simplicity, we have scaled the units such that the material parameter $\param$ is set to unity. 


\subsection{Numerical verification of first order convergence}
\label{sec:numVer}
We verify numerically first order convergence on the unit circle for a smooth radially symmetric solution $\boldsymbol u(r,t) = r^a t^b  \hat{\boldsymbol \phi}$ with the forcing term $\boldsymbol f(r,t) = (b r^a t^{b-1}-((a+1)t^b)^{p-1}r^{(a-1)(p-1)-1})\hat{\boldsymbol \phi}.$ Specifically, the constants $a, b>0$ are parameters to be chosen, $r$ is the radial cylindrical coordinate and $\hat{\boldsymbol \phi}$ is the azimuthal unit vector. Note that by radial symmetry, $\boldsymbol u(r,t)$ is necessarily divergence-free. For these tests, we have fixed $p=5$ and the final time $T$=5e-3.

\begin{figure}[h!]
\centering
	\subfloat[$a=2,b=1$]{\label{figs:a2b1}\includegraphics[width=0.45\textwidth]{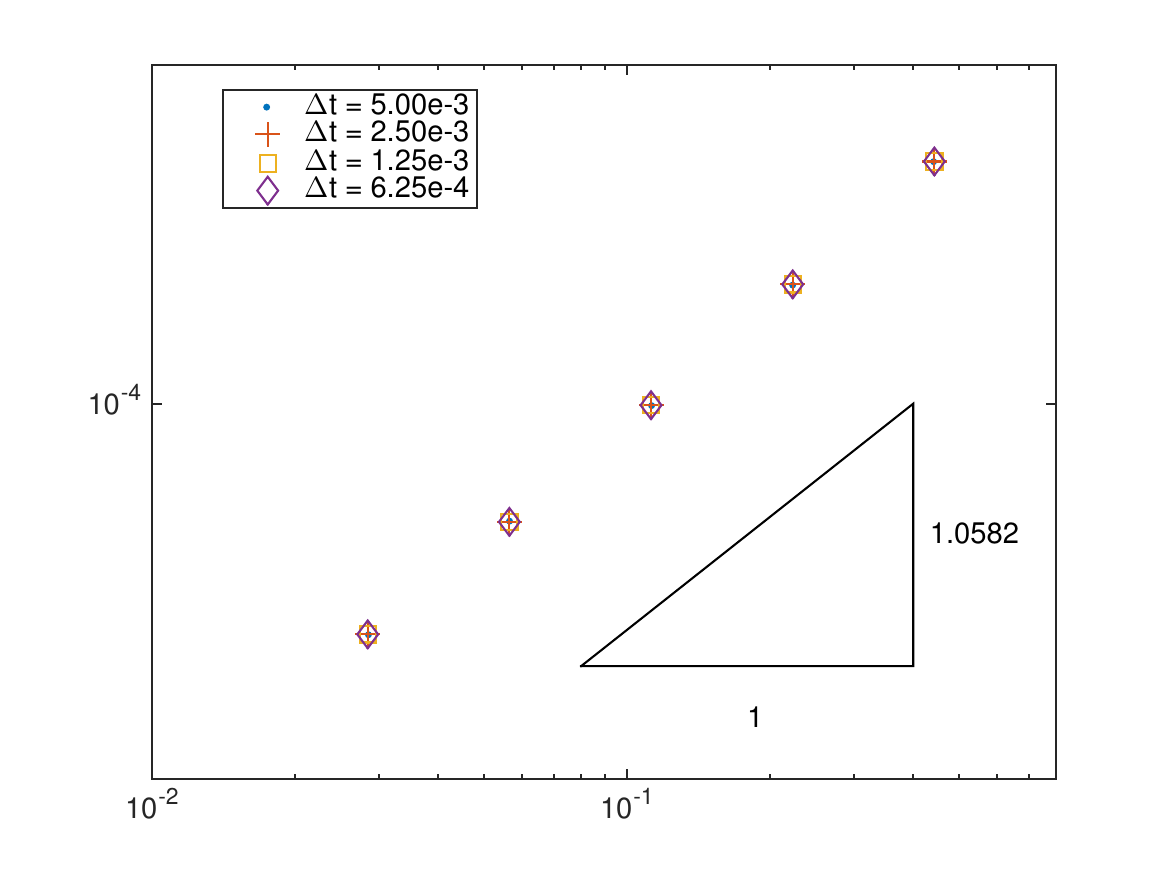}}
	\subfloat[$a=1,b=2$]{\label{figs:a1b2}\includegraphics[width=0.45\textwidth]{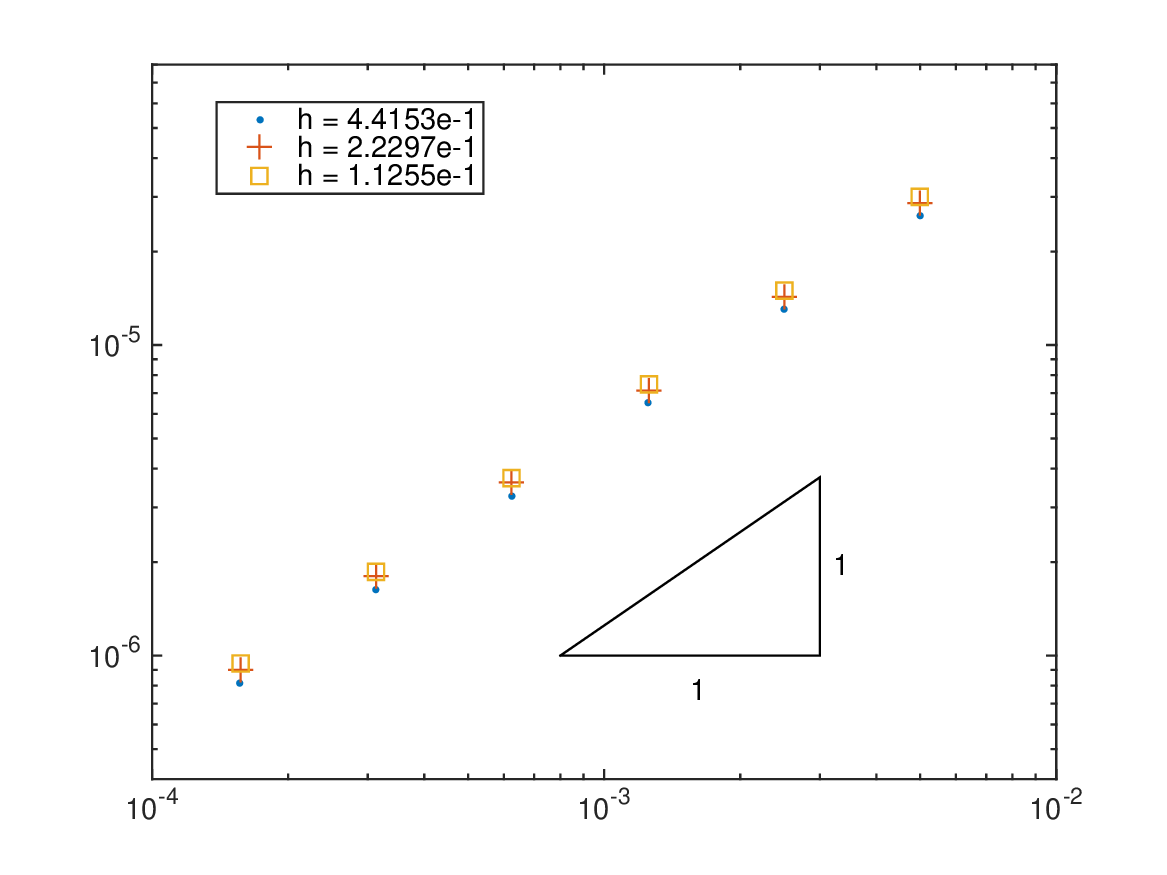}}
	\caption{Plot of $\norm{\bb u - \bb u_h}_{L^2(\Omega)}$ vs $h$ and vs $\Delta t$, respectively.}
\end{figure}

\begin{figure}[h!]
\centering
	\subfloat[$\norm{\bb u - \bb u_h}_{L^2(\Omega)}$ vs $\Delta t$]{\label{figs:a2b2vsdt}\includegraphics[width=0.45\textwidth]{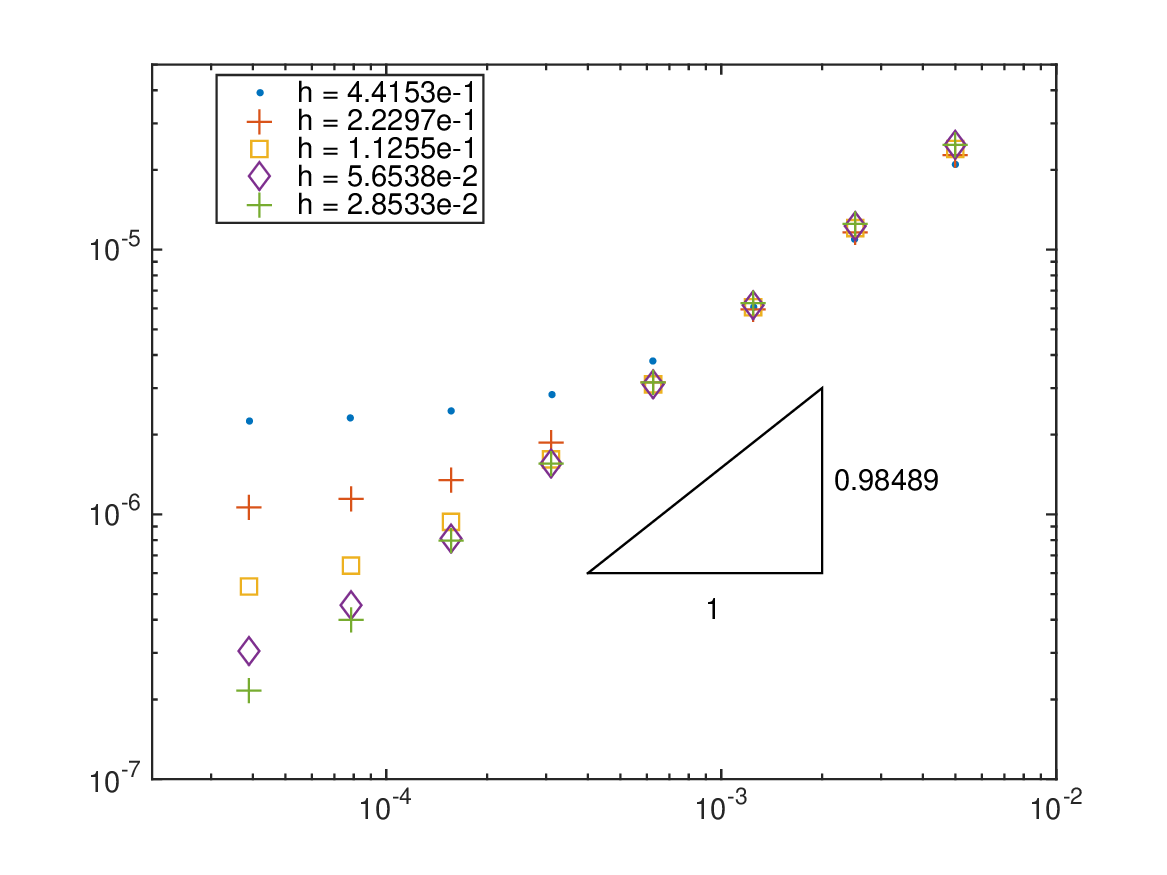}}
	\subfloat[$\norm{\bb u - \bb u_h}_{L^2(\Omega)}$ vs $h$]{\label{figs:a2b2vsh}\includegraphics[width=0.45\textwidth]{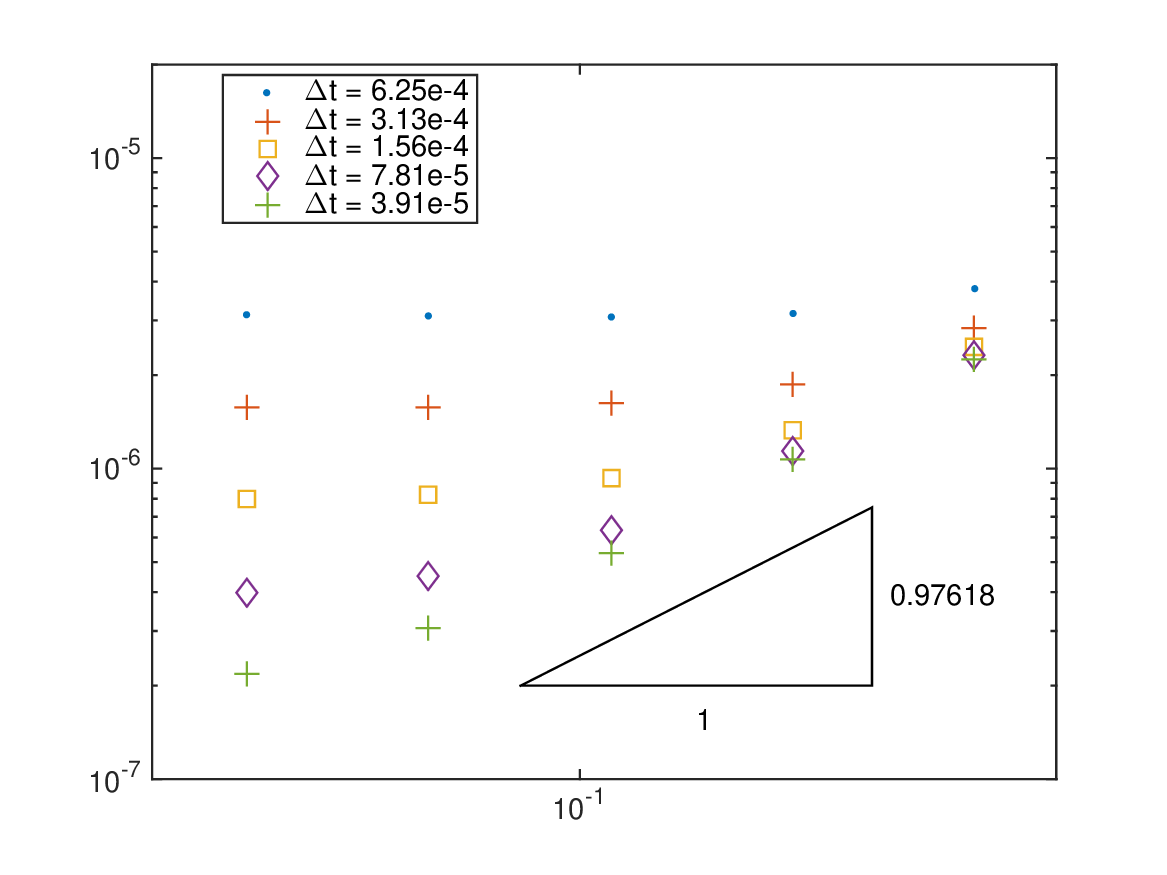}}
	\caption{Plot of error versus $\Delta t$ and $h$ for $a=2, b=2$.}
\end{figure}
For $a=1, b=1$, the solution is linear in both space and time. Since both first order N\'ed\'elec elements and backward Euler method are exact for linear functions, it was observed that the FE solution was accurate up to machine precision.

When $a=2, b=1$, the solution is quadratic in space and linear in time. Thus we expect to only have spatial error of first order in $h$, as shown in Figure \ref{figs:a2b1}. Similarly, for the case $a=1, b=2$, we observed temporal error of first order in $\Delta t$ in Figure \ref{figs:a1b2}.

For $a=2, b=2$, the solution is quadratic in both space and time. From Figure \ref{figs:a2b2vsdt}, first order error in $\Delta t$ was observed in time when the mesh was sufficiently fine. Similarly from Figure \ref{figs:a2b2vsh}, first order error in $h$ was observed in space when the time step size was sufficiently small.

\subsection{Numerical verification of reliability of a posteriori error estimators}

Next, we numerically verify the reliability of the error estimators presented in Section \ref{sec:apee}. We will first look at the case of a convex $C^{1,1}$ domain $\Omega$ given by the unit circle and then consider a nonconvex $C^{1,1}$ domain given by an annulus. Finally, we look at a moving front case with sharp gradient often encounter in practice for the $p$-curl problem. In all cases, the computational mesh $\Omega_h$ was constructed to be an interior mesh of $\Omega$ using the native mesh generator of FEniCS.

\subsubsection{Convex domain - circle}
For the unit circle $\Omega$, we generate an interior mesh $\Omega_h$ by specifying the number of perimeter segments of the polygonal domain to be inscribed inside the unit circle. For instance, if the number of segments is $N$ with equal length and the perimeter vertices lies on the unit circle, then elementary trigonometry shows that $\text{vol}(\Omega_h) = N\sin\left(\frac{\pi}{N}\right)\cos\left(\frac{\pi}{N}\right)$, which converges to $\pi = \text{vol}(\Omega)$ as $N\rightarrow \infty$. In the following, as we refine the mesh by reducing $h$ by half, we correspondingly also double the number of segments $N$ on the perimeter of $\Omega_h$.

On $\Omega$ and $t\in [0,1]$, we employed a radially symmetric inward moving front solution of the form $\boldsymbol u(r,t) = h(r,t) \hat{\boldsymbol \phi}$ with,
$$
h(r,t) = \left\{\begin{split}
(r-1+t)^a, & \hspace{4mm} r>1-t \\
0, & \hspace{4mm} r\leq 1-t
\end{split}\right.,
$$
where $a\geq 1$ is a parameter to be chosen. It can be checked that the current density has the form $\nabla\times \boldsymbol u(r,t) = j(r,t) \hat{\boldsymbol z}$ with
$$
j(r,t) = \left\{\begin{split}
(r-1+t)^{a-1}\left(a+1-\frac{1-t}{r}\right), & \hspace{4mm} r>1-t \\
0, & \hspace{4mm} r\leq 1-t
\end{split}\right..
$$ 
Thus, the corresponding forcing term is given by,
$$\boldsymbol f(r,t) = (h_t(r,t)-(p-1)j(r,t)^{p-2}j_r(r,t) )\hat{\boldsymbol \phi}.$$
The motivation for choosing this family of manufactured solutions originates from an exact analytical solution of Mayergoyz \cite{MikAE13} of the $p$-curl problem in 1D. In particular, it is known that the parameter $a=\frac{p-1}{p-2}$ for the 1D case and so $a\approx 1$ for large values of $p$. Moreover, it can be seen that as $a$ approaches $1$, the current density $j(r,t)$ has steeper gradients and converges pointwise to a discontinuous function. In fact for $t<1$, it can be checked that $j(r,t) \in W^{1,p}(\Omega)$ if and only if $a>2-\frac{1}{p}$.\footnote{Since $j_r \sim s^{a-2}$ where $s$ is the distance away from the front, $j_r \in L^p(\Omega) \Leftrightarrow p(a-2)+1>0$.} Thus, for $a$ close to $1$, we do not expect the FE approximation using N\'ed\'elec elements to be accurate, since its interpolation error requires $\nabla\times \boldsymbol u(r,t)=j(r,t)\hat{\boldsymbol z}$ to be at least a $W^{1,p}(\Omega)$ function \cite[Theorem 1.117]{ErnGue04}. For these reasons, we have focused on a  case satisfying $a>2-\frac{1}{p}$. More specifically, we have fixed $a=3$, $p=25$, $\Delta t$=5e-4. 

The integration in time was computed numerically using the composite midpoint rule. Also note that, since the initial field $\bb u_0(\bb x)=\bb 0\in \Xhz^{(k)}$, the initial error is identically zero. Moreover, recalling that first order N\'ed\'elec elements are element-wise divergence free, we omitted computing $\eta_d$ as it is identically zero. Finally, the boundary elements were refined to be sufficiently small such that the term $h_\partial^{\frac{2p}{p-2}}$ was negligible in the computation of $NC_2^2$.

\begin{figure}[h!]
\centering
	\label{figs:errVsEst}\includegraphics[width=0.5\textwidth]{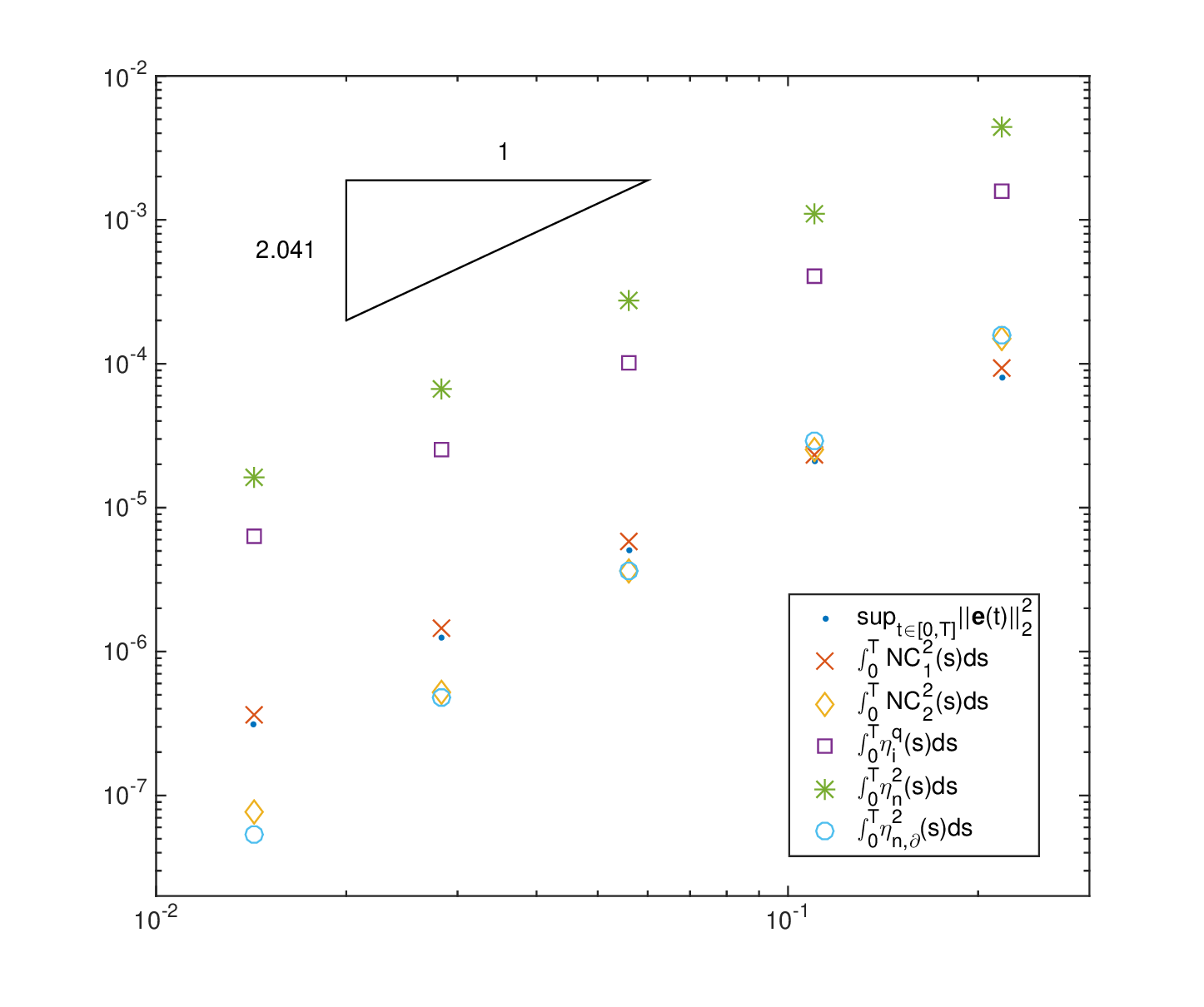}
	\caption{Comparison of error and estimators versus $h$ at $T=$4e-1.}
\end{figure}

In Figure \ref{figs:errVsEst}, the error in $\sup_{s\in[0,T]}\norm{\bb e(s)}^2_{L^2(\Omega)}$, nonconforming geometric error estimators $\int_0^T \NC_1^2(s)ds, \int_0^T \NC_2^2(s)ds$ and estimators $\int_0^T \eta_i^q(s) ds, \int_0^T \eta_n^2(s) ds$, and $\int_0^T \eta_{n,\partial}^2(s) ds$ from Theorem \ref{thm:main_thm2} are plotted for various mesh sizes $h$. Note that we have omitted showing $\int_0^T\norm{\nabla \times\bb e}^p_{L^p(\Omega)}(s) ds$, $\int_0^T \eta_d^2(s) ds$, $\int_0^T \eta_t^q(s) ds$ and $\int_0^T \eta_{t,\partial}^q(s) ds$ as their values were observed to be near machine precision zero due to their small magnitude and/or their dependence on the exponent of $p=25$. As illustrated, we observed quadratic order of convergence in $h$ for both the error and estimators showing agreement of the reliability of the estimators. This is consistent with the first order convergence of Section \ref{sec:numVer}, since the error quantity under consideration is squared with respect to the $L^2$ norm. We also observed that the error estimators $\int_0^T \NC_2^2(s)ds$ and $\int_0^T \eta_{n,\partial}^2(s) ds$ decreases at a faster rate due to the refinement of $\partial\Omega_h$.

In the absence of knowledge on the constants $C_1$ and $C_2$ from Theorem \ref{thm:main_thm2}, we can still measure the reliability of the error estimators by the quantity $\kappa$\footnote{For stationary problems, $\kappa$ is usually called the effectivity index of the error estimators.} defined as the ratio of estimators over the errors by,
$$
\kappa := \frac{\int_0^T \NC_1^2(s)+\NC_2^2(s)+\eta_d^2(s)+\eta_n^2(s)+\eta_{n,\partial}^2(s)+\eta_i^q(s)+\eta_t^q(s)+\eta_{t,\partial}^q(s) ds}{\sup_{s\in[0,T]}\norm{\bb e(s)}^2_{L^2(\Omega)}+\int_0^T\norm{\nabla \times\bb e(s)}^p_{L^p(\Omega)} ds}.
$$ Ideally, for efficient mesh adaptivity, one would like to have $\kappa\approx 1$. However, due to the unknown constants inherent in the present residual type error estimation and the dependence on $T$ due to time integration, we can only expect $\kappa$ to \emph{decrease} with $T$. In particular, since the error estimators from Theorem \ref{thm:main_thm2} are reliable, then $\kappa$ should be bounded below by the constant $\frac{\min\{1,C_1(p,\alpha)\}}{C_2(p,T)}$, where $C_2$ increases in the worst case exponentially with respect to $T$. 

\begin{figure}[h!]
\centering
	\subfloat[$\kappa$ versus $h$]{\label{figs:kappa_h}\includegraphics[width=0.45\textwidth]{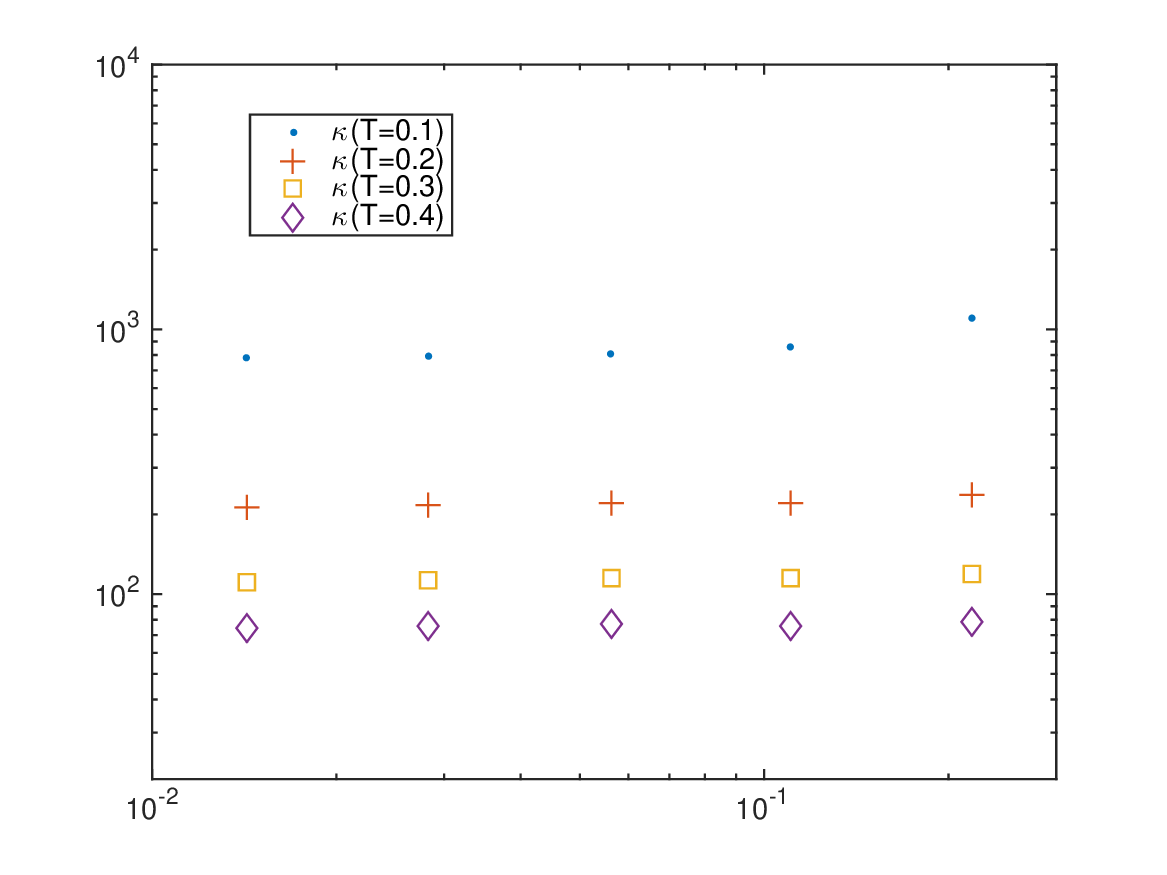}}
	\subfloat[$\kappa$ versus $T$]{\label{figs:kappa_T}\includegraphics[width=0.45\textwidth]{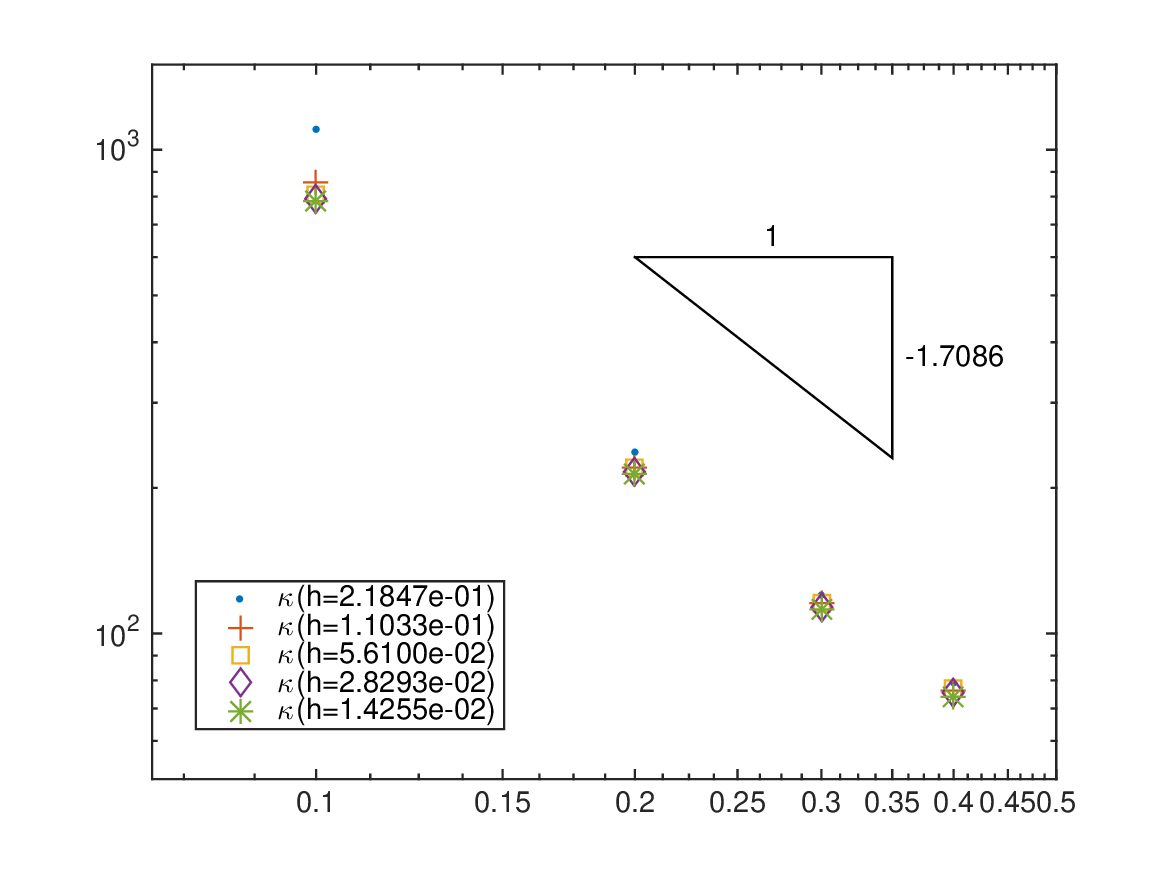}}
	\caption{Comparison of $\kappa$ versus $h$ and $T$.}
\end{figure}

In Figure \ref{figs:kappa_h}, $\kappa$ is shown to be largely independent of $h$ and decreases with $T$. This suggests that the error estimators are comparable to the actual error up to a factor of $\kappa$. Moreover, from Figure \ref{figs:kappa_T}, we see that $\kappa \sim T^{-1.71}$ which suggests that the exponential dependence on $T$ for the constant $C_2$ in Theorem \ref{thm:main_thm2} may be sharpened to $\sim T^{1.71}$ in this case.

%

%

\subsubsection{Nonconvex domain - annulus}

Next we consider a nonconvex $C^{1,1}$ domain given by the annulus region $\Omega = \{\bb x\in\mathbb{R}^2 : 0.5\leq \norm{\bb x}\leq 1\}$. We construct an interior mesh $\Omega_h$ by specifying the number of perimeter segments $N_o$ on the outer radius of $r = 1$ and removed a polygonal region with the number of perimeter segments $N_i$ inscribed on the inner radius $r=R$. In order the guarantee $\Omega_h \subset \Omega$, $R$ was chosen to be slightly \emph{larger} than $0.5$ so that the removed polygonal region covers the removed part of the annulus region of $\{\bb x\in \mathbb{R}^2 : \norm{\bb x}\leq 0.5\}$. For instance, if the perimeter segments on the inner radius is of equal length, then elementary trigonometry shows that the inner radius is $R = 0.5\left(\cos\left(\frac{\pi}{N_i}\right)\right)^{-1}> 0.5$ which converges to $0.5$ as $N_i\rightarrow \infty$. Similar to the unit circle case, as we refine the mesh by reducing $h$ by half, we correspondingly also double the number of segments on both the number of inner and outer perimeter segments $N_i$ and $N_o$ of $\Omega_h$.

On the annulus domain, we used a similar manufactured solution as the unit circle except the radially symmetric moving front solution is moving outward from $r=0.5$. The choice was made to differentiate the inward-moving solution of the unit circle case. Specifically, it has the form $\boldsymbol u(r,t) = h(r,t) \hat{\boldsymbol \phi}$ with,
$$
h(r,t) = \left\{\begin{split}
(t+0.5-r)^a, & \hspace{4mm} r\leq t+0.5 \\
0, & \hspace{4mm} r> t+0.5
\end{split}\right.,
$$where we have again chosen $a=3$, $p=25$ and $\Delta t = 5e-4$.

In Figure \ref{figs:AnnErrVsEst}, we observed similar quadratic order of convergence in $h$ for both the error and estimators. We also observed $\kappa$'s independence of $h$ in Figure \ref{figs:AnnKappa_h} and rate of decrease in $T$ in Figure \ref{figs:AnnKappa_T}, where $\kappa\sim T^{-1.94}$ in this case.

\begin{figure}[h!]
\centering
	\label{figs:AnnErrVsEst}\includegraphics[width=0.5\textwidth]{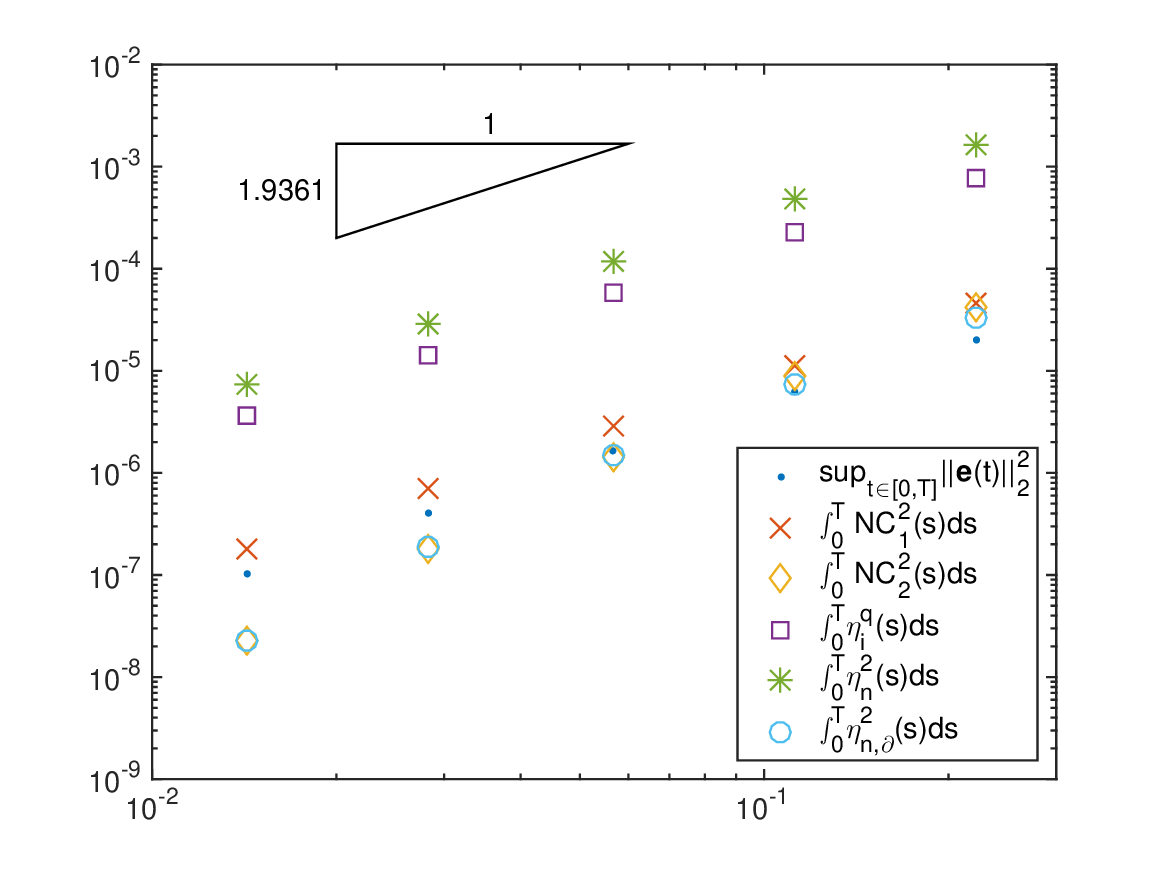}
	\caption{Comparison of error and estimators versus $h$ at $T=$3.2e-1.}
\end{figure}

\begin{figure}[h!]
\centering
	\subfloat[$\kappa$ versus $h$]{\label{figs:AnnKappa_h}\includegraphics[width=0.45\textwidth]{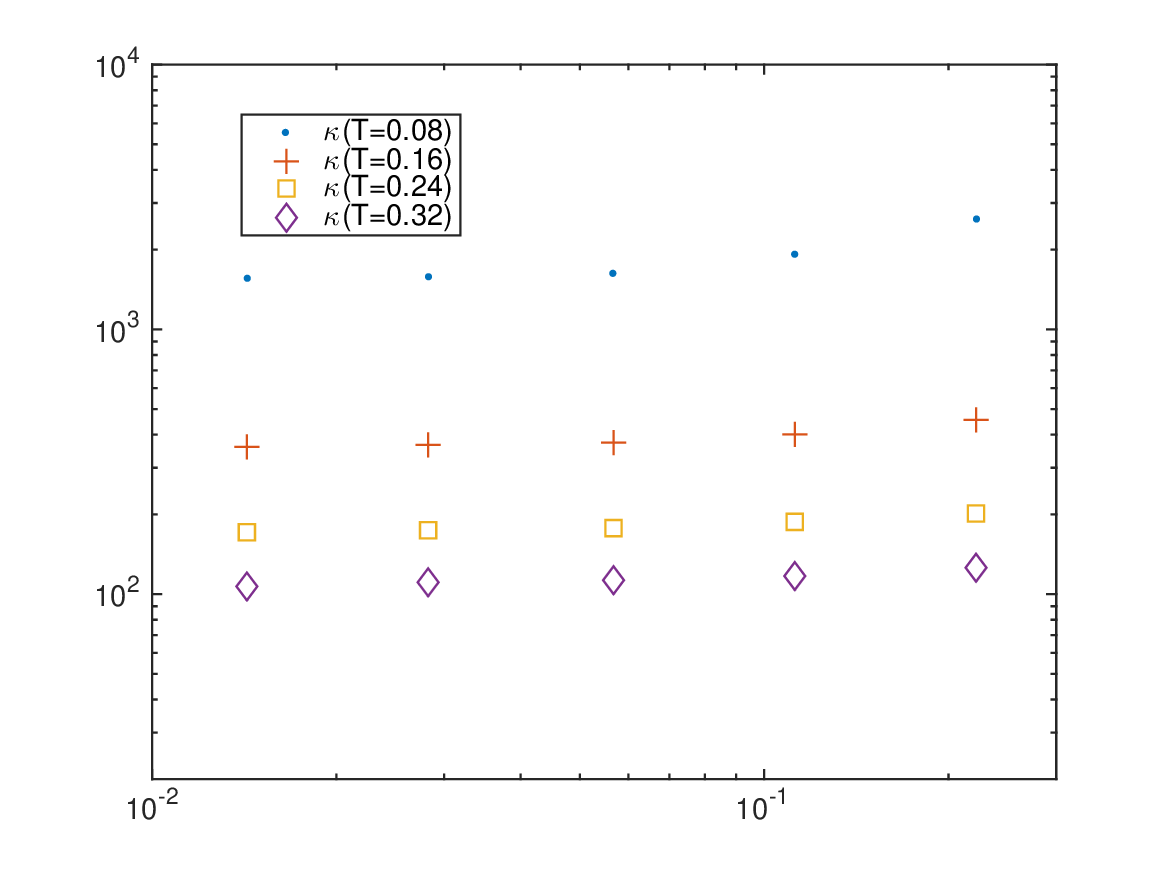}}
	\subfloat[$\kappa$ versus $T$]{\label{figs:AnnKappa_T}\includegraphics[width=0.45\textwidth]{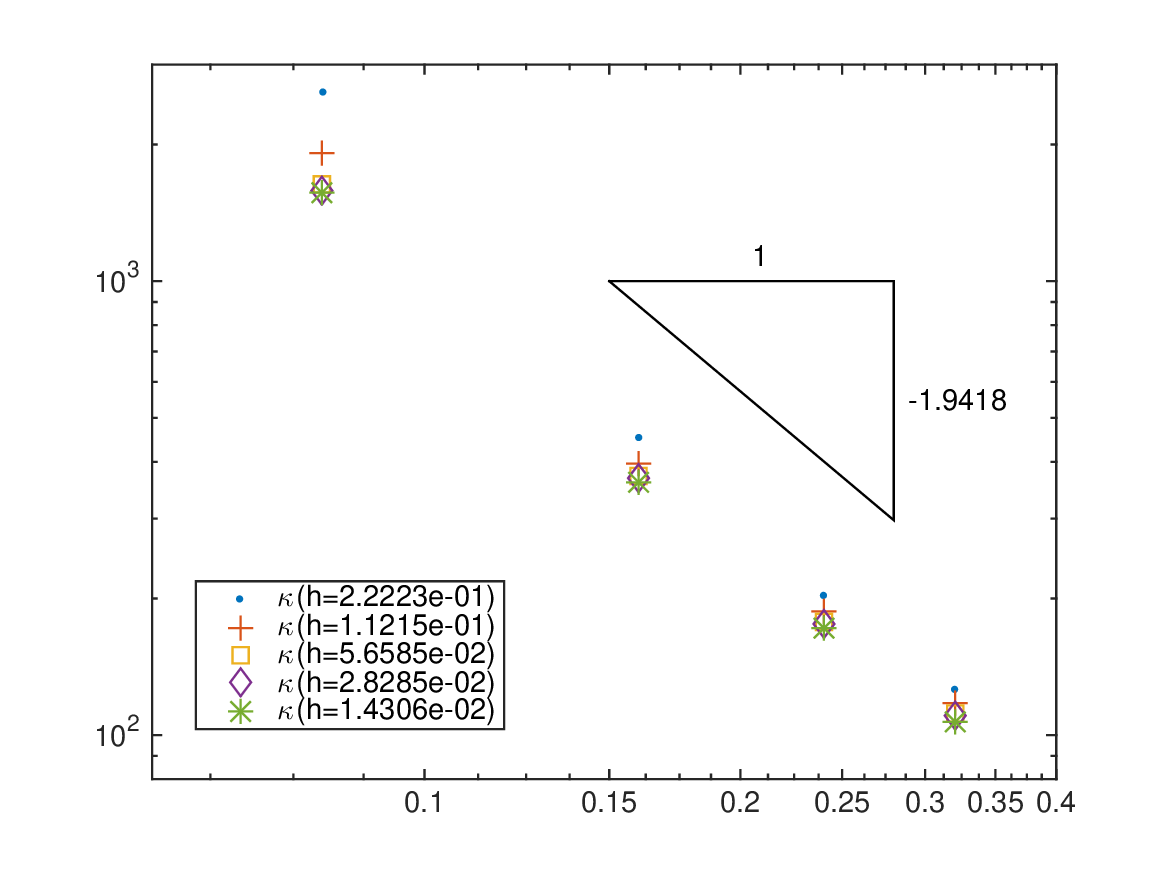}}
	\caption{Comparison of $\kappa$ versus $h$ and $T$.}
\end{figure}

\subsubsection{Nonsmooth case}

Finally, we look at a case for which $\nabla\times \bb u \notin W^{1,p}(\Omega)$. For this, we consider again the manufactured solution on the unit circle domain and we chose $a=1.6$ and $p=10$ so that $a<2-\frac{1}{p}$. The purpose here is to compare qualitatively between the error and estimators even in this nonsmooth case. As illustrated in Figure \ref{figs:jError} and Figure \ref{figs:norEst}, the region where the local estimators $\eta_n$ are largest agrees with regions where the sharp gradient occurs in the current density $\nabla\times \bb u$. Moreover, in Figure \ref{figs:hError} and Figure \ref{figs:intEst}, the local estimators $\eta_i$ identified the boundary region as where the increasing magnetic field $\bb u$ was being applied.

\begin{figure}[h!]
\centering
{\label{figs:hError}\includegraphics[width=0.675\textwidth]{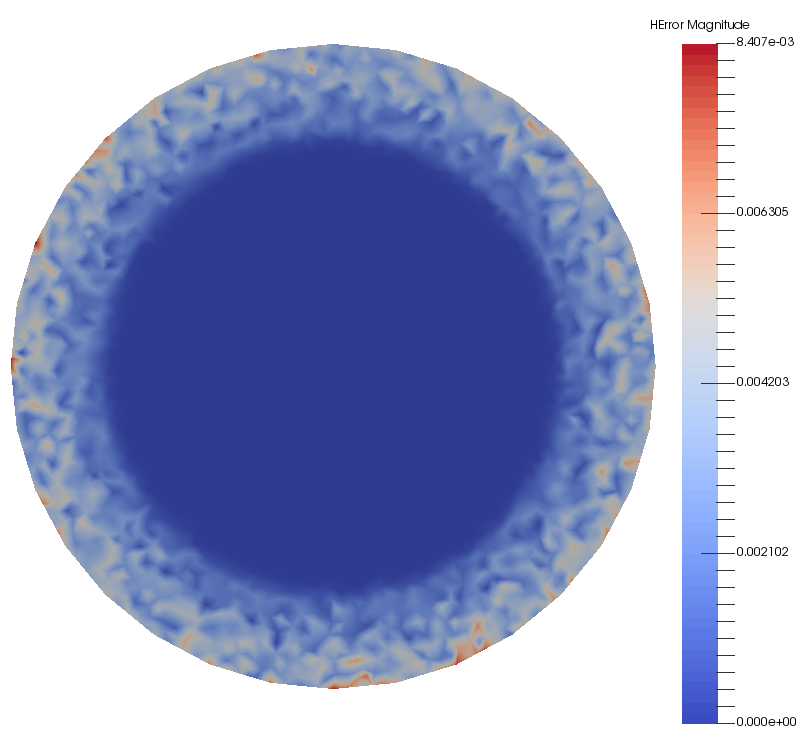}}
\caption{Local $L^2$ error of $\bb u$ at $t=0.272$. The scale represents values between $0$ and $8.4\times 10^{-3}$.}
\end{figure}

\begin{figure}[h!]
\centering
{\label{figs:jError}\includegraphics[width=0.675\textwidth]{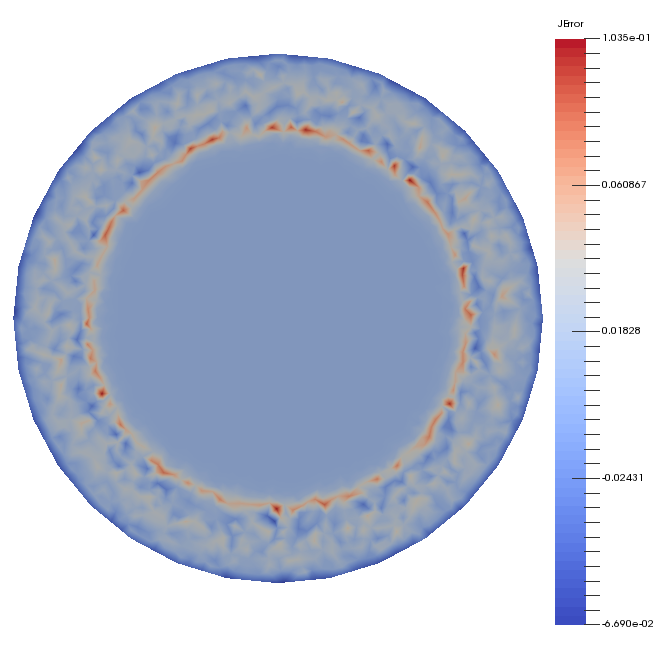}}
\caption{Local $L^2$ error of $\nabla\times \bb u$ at $t=0.272$. The scale represents values between $-6.7\times 10^{-2}$ and $1.0\times 10^{-1}$. Note that the largest errors occur at the moving front and at the boundary of the domain.}
\end{figure}

\begin{figure}[h!]
\centering
{\label{figs:intEst}\includegraphics[width=0.675\textwidth]{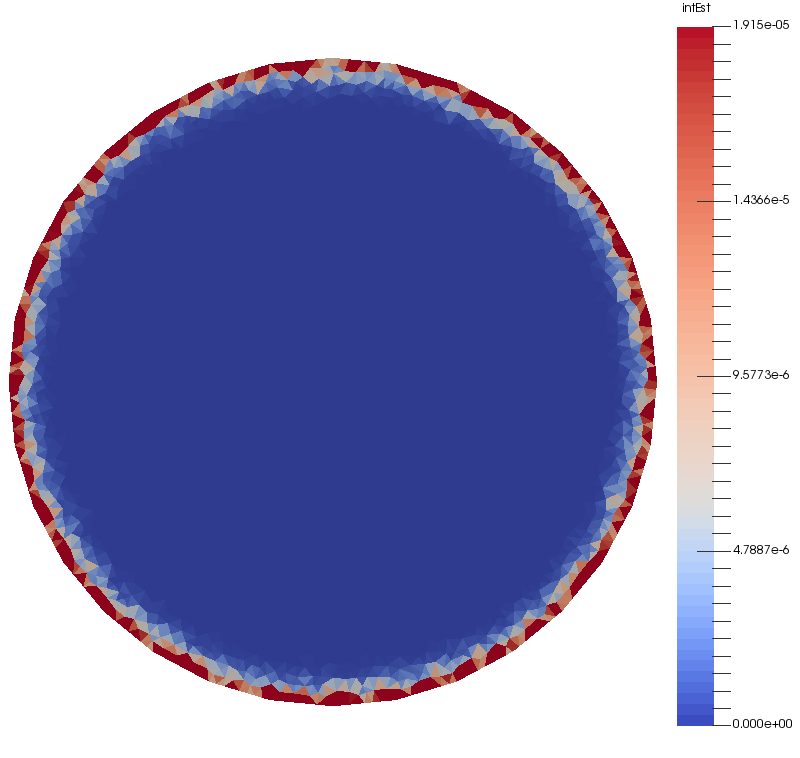}}
\caption{Local estimator $\eta_i$ at $t=0.272$. The scale represents values between $0$ and $1.9\times 10^{-5}$.}
\end{figure}

\begin{figure}[h!]
\centering
{\label{figs:norEst}\includegraphics[width=0.675\textwidth]{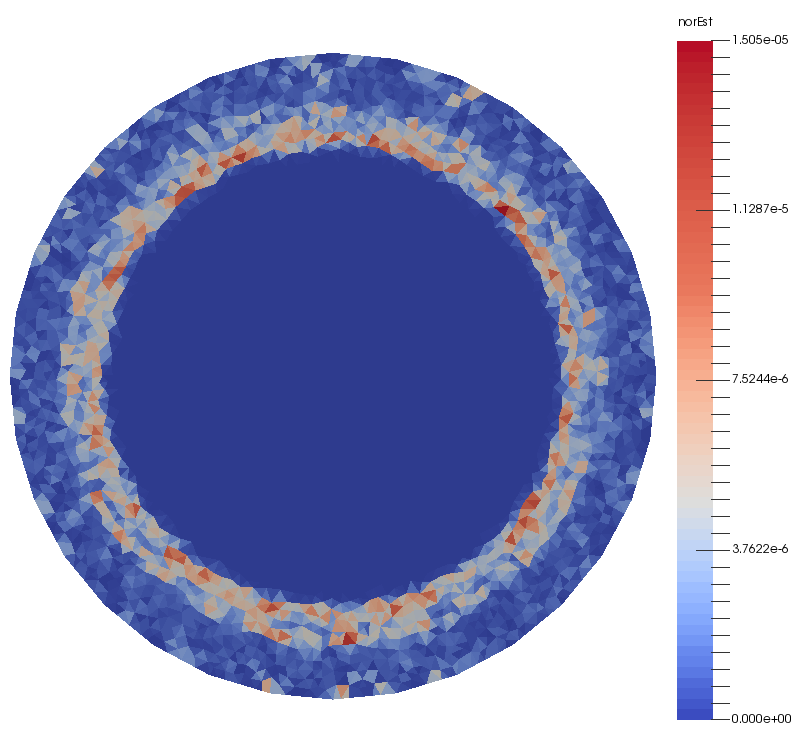}}
\caption{Local estimator $\eta_n$ at $t=0.272$. The scale represents values between $0$ and $1.5\times 10^{-5}$.}
\end{figure}

\section{Conclusion} 

This paper has presented an original a posteriori residual-based error estimator for a nonlinear
wave-like propagation problem modeling strong variations in the magnetic field density inside
high-temperature superconductors. The techniques used circumvent the non-conformity of the 
numerical approximations in a simple manner and the nonlinearities are handled using only
coercive properties of the spatial operator, and without any linearization. Preliminary numerical
results in two space dimensions indicate that the residuals are asymptotically exact, up to a constant. 

An important avenue for future research would be to develop error estimators which are both reliable and efficient. The work of Carstensen, Liu, and Yan on quasi-norms for the $p$-Laplacian appears to be the next natural step, given the similarities in the analytic framework underlying both problems \cite{LiuYan01,CarKlo03,CarLiuYan06a}. We also mention the recent optimality results of Diening and Kreuzer on adaptive finite element methods for the $p$-Laplacian \cite{DieKre08,BelDieKre12} and of Alaoui, Ern and Vohral\^ik on a posteriori error estimates for monotone nonlinear problems \cite{ElAErnVoh11}. 
Moreover, further investigation is needed concerning the efficiency for solving the nonlinear discrete problems arising from successive adaptive mesh based on such error estimators. At the moment, the optimal design of new high temperature superconducting devices is limited by the high computational cost of such simulations, and all means of improving this efficiency should be examined
in hopes of removing this bottleneck.
 
\appendix

\section{Non-homogeneous tangential boundary condition}
\label{app:non-homomBC}
We can account for the non-homogeneous tangential boundary conditions on $\partial \Omega$ by establishing a ``Duhamel's principle" for the $p$-curl problem. The novelty here is in the $L^p$ treatment of the homogeneous auxiliary variables and in the nonlinearity. 

Denote $\Wcurlz= \{\bb v\in \Wcurl:\nabla\times\bb v=0\}$ as the $L^p$ space of curl-free functions. It suffices to show the following:
\begin{theorem}
Let $\Omega$ be a $C^{1,1}$ bounded simply-connected domain in $\mathbb{R}^3$ and let $\boldsymbol g\in \gamma_t(\Wcurl)$ with $2\leq p<\infty$. For any $\boldsymbol u\in \Wcurl\cap\Wdivz$ with $\gamma_t(\boldsymbol u) = \boldsymbol g$, there exists a function $\boldsymbol u_g \in \Wcurlz\cap\Wdivz$ with $\gamma_t(\boldsymbol u_g) = \boldsymbol g$ and a function $\hat{\boldsymbol u}\in \X$ such that $\boldsymbol u = \hat{\boldsymbol u} + \boldsymbol u_g$.
 \label{liftThm}
\end{theorem}

Indeed, if such decomposition exists, since $\boldsymbol u_g$ is curl- and divergence-free, the non-homogeneous $p$-curl problem 
reduces to the homogeneous $p$-curl problem,
\begin{equation}  \label{pcurl-redhomog}
\begin{array}{r@{\hspace{2pt}}c@{\hspace{2pt}}ll}
       \pdt\hat{\bb u}+\curl \left[ \rho(\curl \hat{\bb u})\curl \hat{\bb u}\right]&=&\bb f-\pdt\bb u_g,  &\mbox{ in } I\times\Omega, \\
       \nabla \cdot \hat{\bb u} &= &0,  & \mbox{ in } I\times\Omega,  \\
      \hat{\bb u}(0,\cdot)&=&\bb u_0(\cdot)-\bb u_g(0,\cdot),  & \mbox{ in } \Omega, \\
      \bb n\times \hat{\bb u}&=&0,  &\mbox{ on } I\times\partial \Omega.
\end{array}
\end{equation}
  

\begin{proof}
Given a function $\boldsymbol g\in \gamma_t(W^p(\text{curl};\Omega))$, we construct $\boldsymbol u_g\in \Wcurlz\cap\Wdivz$ in three main steps. 

First, let $\tilde{\boldsymbol u}_g\in W^p(\text{curl};\Omega)$ be such that $\gamma_t(\tilde{\boldsymbol u}_g)=\boldsymbol g$. Such $\tilde{\boldsymbol u}_g$ exists by the surjectivity of the image space $\gamma_t(W^p(\text{curl};\Omega))$. 

Second, let $v \in W^{1,p}_0(\Omega)$ be the solution to the problem:
\begin{equation}
\rpair{\nabla v, \nabla\psi}=\rpair{\tilde{\boldsymbol u}_g,\nabla \psi}, 
                                    \quad \forall \psi\in W^{1,q}_0(\Omega). \label{1stEqn}
\end{equation} Such a function $v$ exists if the following two conditions hold \cite{ErnGue04}:
\begin{equation}
                  0 <    \inf_{0\neq\phi\in W^{1,p}_0(\Omega)} \sup_{0\neq\psi\in W^{1,q}_0(\Omega)}
                               \frac{\rpair{\nabla \phi, \nabla \psi}}{\norm{\phi}_{W^{1,p}_0(\Omega)}\norm{\psi}_{W^{1,q}_0(\Omega)}}, \label{infsup1}
\end{equation} and if for all $\phi \in W^{1,p}_0(\Omega)$,
\begin{equation}
           \rpair{\nabla \phi, \nabla\psi} = 0, \quad \forall \psi \in W^{1,q}_0(\Omega) 
                                   \quad \Rightarrow \quad \phi = 0. \label{nondegen1}
\end{equation}

We first show the inf-sup condition. From the Helmholtz decomposition of Theorem \ref{C1decomp}, for $\boldsymbol v\in W^{1,q}_0(\Omega)^3$, there exists $\boldsymbol z_v\in \Xq$ and $\phi_v\in W^{1,q}_0(\Omega)$ such that $\boldsymbol v=\boldsymbol z_v+\nabla \phi_v$ with $\norm{\boldsymbol z_v}_{L^q}+\norm{\nabla \phi_v}_{L^q}\leq C\norm{\boldsymbol v}_{L^q}$ for some constant $C>0$. In particular, for any $\phi\in W^{1,p}_0(\Omega)$, $\rpair{\nabla \phi, \boldsymbol z_v}=0$. This implies that for any $\phi\in W^{1,p}_0(\Omega)$,
\begin{align*}
 \norm{\phi}_{W^{1,p}_0(\Omega)} 
                                      &= \sup_{0\neq\boldsymbol v\in L^q(\Omega)}
                                                 \frac{\rpair{\nabla \phi, \boldsymbol v}}{\norm{\boldsymbol v}_{L^q(\Omega)}} 
                                        = \sup_{0\neq\boldsymbol v\in L^q(\Omega)}
                                                 \frac{\rpair{\nabla \phi, \nabla \phi_v}}{\norm{\boldsymbol v}_{L^q(\Omega)}} \\ 
                                     &\leq C\sup_{0\neq\boldsymbol v\in L^q(\Omega)}
                                          \frac{\rpair{\nabla \phi, \nabla \phi_v}}{\norm{\nabla \phi_v}_{L^q(\Omega)}} 
                                       \leq C\sup_{0\neq\psi\in W^{1,q}_0(\Omega)}
                                           \frac{\rpair{\nabla \phi, \nabla \psi}}{\norm{\nabla \psi}_{L^q(\Omega)}}.
\end{align*}
Since the norm $\norm{\nabla \psi}_{L^q(\Omega)}$ is equivalent to $\norm{\psi}_{W^{1,q}_0(\Omega)}$ for $\psi \in W^{1,q}_0(\Omega)$, dividing the above inequality by $\norm{\phi}_{W^{1,p}_0(\Omega)}$ and taking the infimum over $\phi\in W^{1,p}_0(\Omega)$ shows the inf-sup condition \eqref{infsup1} is satisfied. 

We now explain why condition \eqref{nondegen1} also holds. For $\psi = \phi \in W_0^{1,2}(\Omega) \subset W_0^{1,q}(\Omega)$, by Poincar\'{e}'s inequality the condition $0 = \rpair{\nabla \phi, \nabla\psi} = \norm{\grad \phi}_{\Ldom{2}}^2$ implies that $\phi=0$ almost everywhere. Thus, a unique solution $v\in W^{1,p}_0(\Omega)$ to \eqref{1stEqn} exists.

Third, let $\boldsymbol w\in \X$ be the solution to the problem:
\begin{equation}
\rpair{\nabla \times \boldsymbol w, \nabla \times \boldsymbol \psi}=\rpair{-\nabla \times\tilde{\boldsymbol u}_g,\nabla\times \boldsymbol\psi}, \forall \boldsymbol\psi\in \Xq. \label{2ndEqn}
\end{equation} Similarly, such a function $\boldsymbol w$ exists if the following two conditions hold:
\begin{equation}
0<\inf_{0\neq\boldsymbol\phi\in \X} \sup_{0\neq\boldsymbol\psi\in \Xq}\frac{\rpair{\nabla \times\boldsymbol \phi, \nabla \times\boldsymbol\psi}}{\norm{\boldsymbol \phi}_{\X}\norm{\boldsymbol\psi}_{\Xq}}, \label{infsup2}
\end{equation} and if for all $\bb \phi\in \X$,
\begin{equation}
\rpair{\nabla \times\boldsymbol \phi, \nabla \times\boldsymbol\psi} = 0, \quad  \forall \bb\psi \in \Xq 
                      \quad \Rightarrow \quad \bb{\phi} = 0. \label{nondegen2}
\end{equation}
By Lemma 5.1 of \cite{AmrSel13}, the inf-sup condition \eqref{infsup2} is satisfied. Moreover, since for $\bb\psi = \bb \phi \in V^2(\Omega) \subset \Xq$, $0 = \rpair{\nabla \times\bb{\phi}, \nabla \times\boldsymbol\psi} 
= \norm{\nabla \times \bb \phi}_{\Ldom{2}}^2$ implies $\bb \phi=0$ a.e. by the equivalence of the semi-norm on $\X$; see Corollary 3.2 of \cite{AmrSel13}. Hence, a unique solution $\boldsymbol w \in \X$ to \eqref{2ndEqn} exists.

Combining these three functions, we define
$$
          \boldsymbol u_g := \boldsymbol w + \tilde{\boldsymbol u}_g -\nabla v \in \Wcurl \, .
$$
Note that $\gamma_t(\boldsymbol u_g) = \gamma_t(\boldsymbol w) + \gamma_t(\tilde{\boldsymbol u}_g) -\gamma_t(\nabla v) = \boldsymbol g$, since $\boldsymbol w \in \X$ and $v\in W^{1,p}_0(\Omega)$. Since $\boldsymbol w\in \X$ is divergence-free, $\nabla \cdot \boldsymbol u_g = \nabla \cdot (\tilde{\boldsymbol u}_g -\nabla v) = 0$ as $v$ satisfies \eqref{1stEqn}. Moreover, $\nabla \times \boldsymbol u_g = \nabla \times (\boldsymbol w+\tilde{\boldsymbol u}_g)=0$ since $\boldsymbol w$ satisfies \eqref{2ndEqn}; i.e. $\boldsymbol u_g \in \Wcurlz\cap \Wdivz$. Thus, $\boldsymbol u_g\in \Wcurlz\cap \Wdivz$ with $\gamma_t(\boldsymbol u_g)=\boldsymbol g$. 

Finally, defining $\hat{\boldsymbol u} := \boldsymbol u - \boldsymbol u_g \in \Wcurl\cap \Wdivz$ and noting $\gamma_t(\hat{\boldsymbol u}) = \gamma_t(\boldsymbol u) - \gamma_t(\boldsymbol u_g) = 0$, $\boldsymbol u - \boldsymbol u_g \in \X$. This shows that $\boldsymbol u = \hat{\boldsymbol u} + \boldsymbol u_g$ as claimed.
\end{proof}

\section*{Acknowledgments}
The authors would like to thank Fr\'ed\'eric Sirois for the countless discussions on the physical model of the $p$-curl problem. Also, the first author would like to thank Gantumur Tsogtgerel for his valuable comments.  Portions of results of Section \ref{sec:apee} and Section \ref{sec:ACloss} were included in the first author's PhD thesis \cite{Wan14}. The authors would also like to thank the anonymous reviewers for their time and helpful suggestions to improve the paper. 

\bibliographystyle{siamplain}
\bibliography{refs}

\end{document}